\documentclass[a4paper,10pt]{amsart}

\usepackage[bookmarks]{hyperref}
\usepackage{amsfonts}
\usepackage{amsmath}
\usepackage{amssymb,amsxtra,amscd,amsthm}
\usepackage{mathrsfs}
\usepackage{mathtools}
\usepackage{todonotes}
\usepackage[foot]{amsaddr}
\usepackage{comment}
\newcommand{\R}{\mathbb{R}}
\newcommand{\C}{\mathbb{C}}

\newcommand{\N}{\mathbb{N}}

\newcommand{\supp}{\operatorname{supp}\,}

\newcommand{\bv}{\operatorname{bv}}

\newtheorem{theorem}{Theorem}[section]
\newtheorem{lemma}[theorem]{Lemma}
\newtheorem{corollary}[theorem]{Corollary}
\newtheorem{proposition}[theorem]{Proposition}

\theoremstyle{definition}
\newtheorem{remark}[theorem]{Remark}
\newtheorem{definition}[theorem]{Definition}

\allowdisplaybreaks

\makeatletter
\DeclareRobustCommand\widecheck[1]{{\mathpalette\@widecheck{#1}}}
\def\@widecheck#1#2{%
    \setbox\z@\hbox{\m@th$#1#2$}%
    \setbox\tw@\hbox{\m@th$#1%
       \widehat{%
          \vrule\@width\z@\@height\ht\z@
          \vrule\@height\z@\@width\wd\z@}$}%
    \dp\tw@-\ht\z@
    \@tempdima\ht\z@ \advance\@tempdima2\ht\tw@ \divide\@tempdima\thr@@
    \setbox\tw@\hbox{%
       \raise\@tempdima\hbox{\scalebox{1}[-1]{\lower\@tempdima\box
\tw@}}}%
    {\ooalign{\box\tw@ \cr \box\z@}}}
\makeatother

\usetikzlibrary{arrows,chains,matrix,positioning,scopes}
\makeatletter
\tikzset{join/.code=\tikzset{after node path={%
\ifx\tikzchainprevious\pgfutil@empty\else(\tikzchainprevious)%
edge[every join]#1(\tikzchaincurrent)\fi}}}
\makeatother
\tikzset{>=stealth',every on chain/.append style={join},
         every join/.style={->}}
\tikzstyle{labeled}=[execute at begin node=$\scriptstyle,
   execute at end node=$]

\begin{document}
\title[Boundary values of zero solutions of differential operators]{Boundary values of zero solutions of hypoelliptic  differential operators in ultradistribution spaces}

\author[A. Debrouwere]{A. Debrouwere$^1$}
\address{$^1$Department of Mathematics: Analysis, Logic and Discrete Mathematics, Ghent University, Krijgslaan 281, 9000 Gent, Belgium}
\email{andreas.debrouwere@UGent.be}

\author[T.\ Kalmes]{T.\ Kalmes$^2$}
\address{$^2$Technische Universit\"at Chemnitz, Faculty of Mathematics, 09107 Chemnitz, Germany}
\email{thomas.kalmes@math.tu-chemnitz.de}

\thanks{This version of the article has been accepted for publication in Mathematische Annalen, after peer review but is not the Version of Record and does not reflect post-acceptance improvements, or any corrections. The Version of Record is available online at: \href{http://dx.doi.org/10.1007/s00208-022-02411-x}{http://dx.doi.org/10.1007/s00208-022-02411-x}}

\begin{abstract}
	We study ultradistributional boundary values of zero solutions of a hypoelliptic constant coefficient partial differential operator $P(D) = P(D_x, D_t)$ on $\R^{d+1}$. Our work unifies and considerably extends various classical results of Komatsu and Matsuzawa about boundary values of holomorphic functions, harmonic functions and  zero solutions of the heat equation in ultradistribution spaces. We also give new  proofs of several results of Langenbruch \cite{Langenbruch1978} about distributional boundary values of zero solutions of $P(D)$.\\
	
	\noindent Keywords: Hypoelliptic and semi-elliptic differential operators; Boundary values; Ultradistributions.\\
	
	\noindent MSC 2020: 35G15; 35H10; 46F05; 46F20
\end{abstract}

\maketitle

\section{Introduction}
The study of distributional and ultradistributional boundary values of holomorphic functions, which goes back to the seminal works of K\"othe \cite{Kothe} and Tillmann \cite{Tillmann}  for distributions and  of Komatsu \cite{Komatsu} for ultradistributions, is an important subject in the theory of generalized functions. We refer to  the survey article \cite{Meise} and the books \cite{CaKaPi,CaMi} and the references therein for an account of results on this topic. Similarly, the boundary value behavior of harmonic functions \cite{Alvarez,E-K,Komatsu1991,V-V} and of zero solutions of the heat equation \cite{C-K,K-C-K,Matsuzawa-1,Matsuzawa} in generalized function spaces have been thoroughly  investigated.

In the distributional case, Langenbruch \cite{Langenbruch1978} generalized the above results by developing a theory of distributional boundary values for the zero solutions of a hypoelliptic partial differential operator  $P(D) = P(D_x, D_t)$ on $\R^{d+1}$ with constant coefficients. For ultradistributions much less is known in this general setting: In \cite{Langenbruch1979,Langenbruch1979-2} Langenbruch studied spaces of formal boundary values of zero solutions of $P(D)$, in the style of Bengel's approach to hyperfunctions \cite{Bengel}. In \cite[Satz 2.4]{Langenbruch1979-2} he  characterized the zero solutions $f$ of $P(D)$ that admit a boundary value in a given ultradistribution space in terms of the ultradistributional extendability properties of $f$ (see also \cite[Section 3]{Komatsu1979}).  Furthermore, he proved that,   for  certain semi-elliptic operators $P(D)$,  the space of formal boundary values of solutions of $P(D)$ may be identified with the  Cartesian product of  some Gevrey ultradistribution spaces of Roumieu type \cite[Satz 4.9 and Satz 4.10]{Langenbruch1979-2}. The aim of the present paper is to complement these results by extending Langenbruch's results from \cite{Langenbruch1978} for distributions to the framework of ultradistributions, both of Beurling and Roumieu type, defined via weight sequences \cite{Komatsu}.

We now describe the content of the paper and state a sample of our main results. For the sake of clarity, we consider here only Gevrey ultradistribution spaces. Let $P(D)$ be a hypoelliptic partial differential operator  on $\R^{d+1}$.  Let $X \subseteq \R^d$ be open. For $\sigma > 1$ we write $\mathscr{D}^{(\sigma)}(X)$ and $\mathscr{D}^{\{\sigma\}}(X)$ for the spaces of compactly supported Gevrey ultradifferentiable functions of order $\sigma$ of Beurling type and of Roumieu type, and endow these spaces with their natural locally convex topology. The spaces of  Gevrey ultradistributions  of order $\sigma$ of Beurling type and of Roumieu type are defined as the strong dual spaces of $\mathscr{D}^{(\sigma)}(X)$ and $\mathscr{D}^{\{\sigma\}}(X)$. We denote them by $\mathscr{D}'^{(\sigma)}(X)$ and $\mathscr{D}'^{\{\sigma\}}(X)$.  We use $\mathscr{D}'^{[\sigma]}(X)$ as a common notation for $\mathscr{D}'^{(\sigma)}(X)$ and $\mathscr{D}'^{\{\sigma\}}(X)$; a similar convention will be used for other spaces as well. Let $V \subseteq \R^{d+1}$ be open such that $V \cap \R^d = X$. Set $C^\infty_P(V \backslash X) = \{ f \in C^\infty(V \backslash X) \, | \, P(D) f = 0\}$. The boundary value $\bv(f) \in \mathscr{D}'^{[\sigma]}(X)$ of an element $f \in C^\infty_P(V \backslash X)$ is defined as
$$
\langle \bv(f), \varphi \rangle := \lim_{t,s \to 0^+} \int_{X} (f(x,t ) - f(x,-s))  \varphi(x) dx, \qquad \varphi \in \mathscr{D}^{[\sigma]}(X),
$$
provided that  $\bv(f)  \in \mathscr{D}'^{[\sigma]}(X)$ exists. We have two main goals: Firstly, we wish to characterize the elements of $C^\infty_P(V \backslash X)$ that admit a boundary value in  $\mathscr{D}'^{[\sigma]}(X)$ in terms of their local growth properties near $X$. Secondly, we aim to study the uniqueness and the solvability of the Cauchy type problem 
\[ \begin{cases}
	P(D) f = 0& \mbox{ on $V \backslash X$}, \\
	\bv(D^j_tf) = T_j & \mbox{on $X$ for $j = 0, \ldots, m-1$},\\
\end{cases}\]
where $f \in C^\infty(V \backslash X)$, $m = \deg_t P$ and the initial data $T_0, \ldots, T_{m-1}$ belong to  $\mathscr{D}'^{[\sigma]}(X)$. This will lead to the representation of the $m$-fold Cartesian product of $\mathscr{D}'^{[\sigma]}(X)$ by means of boundary values of elements of $C^\infty_P(V \backslash X)$.  

Let us explicitly state the above two results for certain semi-elliptic partial differential operators. To this end, we introduce the following weighted spaces of zero solutions of $P(D)$
\begin{align*}
	C^\infty_{P, (\sigma)}(V\backslash X) := \{ &f \in C^\infty_P(V\backslash X) \, | \, \forall K \Subset V \, \exists h > 0 \, : \, \\
	&\| f\|_{K,h} := \sup_{(x,t) \in K \backslash K \cap \R^d} |f(x,t)| e^{-h|t|^{-\frac{1}{\sigma-1}}} < \infty \},
\end{align*}
$$
C^\infty_{P, \{\sigma\}}(V\backslash X) := \{ f \in C^\infty_P(V\backslash X) \, | \, \forall K \Subset V \, \forall h > 0 \, : \, \| f \|_{K,h} < \infty \},
$$
where $K\Subset V$ denotes that $K$ is a compact subset of $V$. We endow these spaces with their natural locally convex topology. We then have:

\begin{theorem}\label{cor:semi-elliptic operators}
	Let $P(D)$ be a semi-elliptic  partial differential operator on $\R^{d+1}$ with $\operatorname{deg}_{x_1}P=\cdots=\operatorname{deg}_{x_d}P=n$. Set $m = \deg_t P$ and $a_0 = n/m$.  Let $X \subseteq \R^d$ be open and let $V \subseteq \R^{d+1}$ be open such that $V \cap \R^d = X$.
	\begin{itemize}
		\item[$(a)$] Let $\sigma > \max\{1, 1/a_0\}$. For $f \in C^\infty_P(V\backslash X)$ the following statements are equivalent:
		\begin{itemize}
			\item[$(i)$]  $f \in C^\infty_{P,[a_0\sigma]}(V\backslash X)$.
			\item[$(ii)$] $\operatorname{bv}(f) \in \mathscr{D}'^{[\sigma]}(X)$ exists.
			\item[$(iii)$] For every relatively compact  open subset $Y$ of $X$ there is $r>0$ such that $\{ f(\, \cdot \,, t) \, | \, 0 < |t| < r \}$  is bounded in  $\mathscr{D}'^{[\sigma]}(Y)$.
		\end{itemize}
		In such a case, $\operatorname{bv}(D^l_tf) \in \mathscr{D}'^{[\sigma]}(X)$ exists for all $l \in \N_0$. Next, define
		\begin{equation}
			\bv^m: C^\infty_{P, [ a_0\sigma ]}(V\backslash X) \rightarrow  \prod_{j = 0}^{m-1}\mathscr{D}'^{[\sigma]}(X), \, f \mapsto (\bv(D^j_t f))_{0 \leq j \leq m-1}. 
			\label{bv-intro}
		\end{equation}
		Then, the sequence 
		$$
		0  \xrightarrow{\phantom{\phantom,\operatorname{bv}^m\phantom,}} C^\infty_P(V)  \xrightarrow{\phantom{\phantom,\operatorname{bv}^m\phantom,}} C^\infty_{P,[a_0\sigma ]}(V\backslash X) \xrightarrow{\phantom,\operatorname{bv}^m\phantom,} \prod_{j = 0}^{m-1}\mathscr{D}'^{[\sigma]}(X)  \xrightarrow{\phantom{\phantom,\operatorname{bv}^m\phantom,}} 0
		$$
		is exact and $\operatorname{bv}^m$ is a topological homomorphism, provided that $V$ is $P$-convex for supports. In particular, this holds for all open sets of the form $V = X \times I$, with $I \subseteq \R$ an open interval containing $0$.
		\item[$(b)$] Let $\sigma = 1/a_0 > 1$.   We have that $\operatorname{bv}(f) \in \mathscr{D}'^{\{\sigma\}}(X)$ exists for all $f \in C^\infty_P(V\backslash X)$. Define the mapping $\bv^m$ similarly as in \eqref{bv-intro}. Then, the sequence 
		$$
		0  \xrightarrow{\phantom{\phantom,\operatorname{bv}^m\phantom,}} C^\infty_P(V)  \xrightarrow{\phantom{\phantom,\operatorname{bv}^m\phantom,}} C^\infty_{P}(V\backslash X) \xrightarrow{\phantom,\operatorname{bv}^m\phantom,} \prod_{j = 0}^{m-1}\mathscr{D}'^{\{\sigma\}}(X)  \xrightarrow{\phantom{\phantom,\operatorname{bv}^m\phantom,}} 0
		$$
		is exact and $\operatorname{bv}^m$ is a topological homomorphism, provided that $V$ is $P$-convex for supports. In particular, this holds for all open sets of the form $V = X \times I$, with $I \subseteq \R$ an open interval containing $0$.
	\end{itemize}
\end{theorem}

Since a partial differential operator $P(D)$ on $\R^{d+1}$ is elliptic precisely when it is semi-elliptic with $\operatorname{deg}_{x_1}P=\cdots=\operatorname{deg}_{x_d}P= \deg_t P$, Theorem \ref{cor:semi-elliptic operators}$(a)$ is applicable to every elliptic partial differential operator. Moreover, it is well-known that every open set $V \subseteq \R^d$ is $P$-convex for supports in such a case.  Therefore, Theorem \ref{cor:semi-elliptic operators}$(a)$ comprises two classical results of Komatsu about the boundary values of holomorphic functions \cite[Section 11]{Komatsu} and harmonic functions \cite[Chapter 2]{Komatsu1991} in ultradistribution spaces. Also the heat operator and, more generally, the $k$-parabolic operators in the sense of Petrowsky  \cite[Definition 7.11]{Treves} satisfy the assumptions of  Theorem \ref{cor:semi-elliptic operators}$(a)$. In particular, this result  extends Matsuzawa's characterization \cite[Theorem 2.1]{Matsuzawa} of compactly supported ultradistributions via boundary values of zero solutions of the heat equation to general ultradistributions. Theorem \ref{cor:semi-elliptic operators}$(b)$ states that the space $\prod_{j = 0}^{m-1}\mathscr{D}'^{\{1/a_0\}}(X)$ may be identified with the space $C^\infty_{P}(V\backslash X)/ C^\infty_{P}(V)$ of formal boundary values. This partially covers the above mentioned result  of Langenbruch \cite[Satz 4.9 and Satz 4.10]{Langenbruch1979-2} (the assumption $\operatorname{deg}_{x_1}P=\cdots=\operatorname{deg}_{x_d}P$ is not needed  there as Langenbruch considers anisotropic Gevrey ultradistribution spaces).

Next, we  comment on the main new technique used in this article. 
In \cite[p.\ 64]{HoermanderPDO1} H\"ormander  showed the existence of distributional boundary values of holomorphic functions by combining Stokes' theorem with  almost analytic extensions. Petzsche and Vogt \cite{P-V} (see also \cite{Petzsche1984}) extended this method to ultradistributions  by using descriptions of ultradifferentiable classes via almost analytic extensions \cite{Dynkin,R-S19,Petzsche1984,P-V}. We develop here a similar technique to establish the existence of ultradistributional boundary values of zero solutions of a hypoelliptic partial differential operator $P(D)$. Namely, we combine  a Stokes type theorem for $P(D)$ (Lemma \ref{lemma-Green}) with the description of tuples $(\varphi_0, \ldots, \varphi_{m-1})$ of compactly supported ultradifferentiable functions  via functions $\Phi \in \mathscr{D}(\R^{d+1})$ that are almost zero solutions of $P(D)$ and satisfy $D_t^j\Phi( \,\cdot \,,0)=\varphi_j$ for  $j=0,\ldots, m-1$  (Proposition \ref{almost-zero}).  
Petzsche \cite{Petzsche1984} constructed almost analytic extensions of ultradifferentiable functions by means of modified Taylor series. We use here the same basic idea, starting from a power series Ansatz $\Phi$ that formally solves the Cauchy problem $P(D)\Phi = 0$ on $\R^{d+1} \backslash \R^d$ and $D_t^j\Phi( \,\cdot \,,0)=\varphi_j$ on $\R^d$ for $j=0,\ldots, m-1$ \cite[Section 4]{Kalmes18-2}. 

This paper is organized as follows. In the preliminary Section \ref{sect:Diff-oper} we collect several results about partial differential operators and ultradistributions that will be used throughout this work. Next, in Section \ref{sect:FS},  we  construct a fundamental solution of a hypoelliptic partial differential operator $P(D)$ satisfying precise regularity and growth properties. This fundamental solution will play an essential role in our work. We obtain it by a careful analysis of the construction of a fundamental solution of $P(D)$ by Langenbruch \cite[p.\ 12--14]{Langenbruch1978}.  Section \ref{sect:almost} is devoted to the description of ultradifferentiable classes via almost zero solutions of $P(D)$. By combining this description with a Stokes type theorem for $P(D)$, we show in Section \ref{sect-bv} that zero solutions of $P(D)$ satisfying certain growth estimates near $\R^d$ have  ultradistributional boundary values. We also establish the continuity of the boundary value mapping here.
Our main results are proven in Section \ref{sec:Main results}. We adapt several classical techniques used in the study of ultradistributional boundary  values of holomorphic functions and harmonic functions \cite{Komatsu,Komatsu1991}. In Sections \ref{sect:almost}--\ref{sec:Main results} we also indicate how our methods may be adapted to the distributional case, thereby providing new proofs of several results of Langenbruch \cite{Langenbruch1978}. Most notably, we obtain an elementary proof of the continuity of the boundary value mapping (compare with \cite[Section 3]{Langenbruch1978}). Finally, in Section \ref{sect-examples}, we discuss our results for  semi-elliptic operators and prove Theorem \ref{cor:semi-elliptic operators}.

\section{Preliminaries}\label{sect:Diff-oper}
\subsection{Partial differential operators}\label{subsect:Diff-oper}
Elements of $\R^{d+1}=\R^d\times\R$ are denoted by $(x,t)=(x_1,\ldots,x_d,t)$. We often identify $\R^d$  with the subspace $\R^d\times\{0\}$ of $\R^{d+1}$. We write $D = (D_x, D_t)$, where $D_x=-i(\partial_{x_1},\ldots,\partial_{x_d})$ and $D_t=-i\partial_t$. Throughout this article $P\in \C[x_1,\ldots,x_d,t]$ always stands for a non-constant hypoelliptic polynomial. Moreover, we use the notation 
$$
P(x,t)=\sum_{k=0}^m Q_k(x)t^k 
$$
for suitable $m \in \N$ and polynomials $Q_k\in\C[X_1,\ldots,X_d]$, $0 \leq k \leq m$, with $Q_m \neq 0$. Then, $Q_m=c\in\C\backslash\{0\}$ \cite[Example 11.2.8]{HoermanderPDO2} and we assume without loss of generality that $c=1$. As customary, we associate to $P$ the constant coefficient partial differential operator $P(D)$ given by
$$
P(D)=P(D_x,D_t)=\sum_{k=0}^m Q_k(D_x)D_t^k.
$$
Given $V \subseteq \R^{d+1}$  open,  we define  $C^\infty_P(V) = \{ f \in C^\infty(V) \, | \, P(D) f = 0\}$ and endow this space with the relative topology induced by $C^\infty(V)$. 

We now introduce various indices of $P$  that play an important role in this article. 
\begin{definition}\label{gevrey}
	We define $(\gamma_0, \mu_0)  = (\gamma_0(P), \mu_0(P))$  as the largest pair of positive numbers satisfying the following property:  There are $C,R > 0$ such that for all $\alpha \in \N^d_0$ and $l \in \N_0$
	$$
	(|x|^{\gamma_0} + |t|^{\mu_0})^{|\alpha| + l } |D^\alpha_xD^l_t P(x,t)| \leq C |P(x,t)|, \qquad |(x,t)| \geq R.
	$$
\end{definition}

The existence of the pair $(\gamma_0,\mu_0)$ follows from the fact that $P$ is hypoelliptic; see \cite[Section 11.4]{HoermanderPDO2} and \cite[Section 7.4]{Treves}. Furthermore, it holds that $\gamma_0, \mu_0 \leq 1$. 
Let $V \subseteq \R^{d+1}$ be open. For  $\sigma, \tau \geq 1$ we denote by $\Gamma^{\sigma,\tau}(V)$ the space consisting of all $\varphi \in C^\infty(V)$ such that for all $K \Subset V$ there is $h > 0$ such that
$$
\sup_{(\alpha,l) \in \N^{d+1}_0} \sup_{(x,t) \in K} \frac{|D^\alpha_xD^l_t\varphi(x,t)|}{h^{|\alpha| + l}|\alpha|!^{\sigma}l!^{\tau} } < \infty. 
$$
By \cite[Theorem 7.4]{Treves}, we have that $C^\infty_P(V) \subset \Gamma^{1/\gamma_0,1/\mu_0}(V)$.
\begin{lemma} \label{lemma-a_0}
	Set 
	$$
	d(x) = \min \{ |\operatorname{Im} \zeta| \, | \, \zeta \in \C, P(x,\zeta) = 0 \}, \qquad x \in \R^d.
	$$
	For $a > 0$ the following two statements are equivalent:
	\begin{itemize}
		\item[$(i)$] There are $C,R > 0$ such that for all $l \in \N_0$
		$$
		|x|^{al} |D^l_t P(x,t)| \leq C |P(x,t)|, \qquad |x| \geq R, t \in \R.
		$$
		\item[$(ii)$] There are $C,R > 0$ such that 
		$$
		|x|^a \leq Cd(x), \qquad |x| \geq R.
		$$
	\end{itemize}
\end{lemma}
\begin{proof}
	The proof is similar to the one of \cite[Lemma 11.1.4]{HoermanderPDO2} and therefore left to the reader.
\end{proof}

\begin{definition}\label{def:a_0}
	We define $a_0 = a_0(P)$ as the largest positive number $a$ such that conditions $(i)$ and $(ii)$ from Lemma \ref{lemma-a_0} are valid. 
\end{definition}
The existence of $a_0$ follows from the fact that $P$ is hypoelliptic and \cite[Corollary A.2.6]{HoermanderPDO2} (cf.\ the proof of \cite[Theorem 11.1.3]{HoermanderPDO2}). Note that $\gamma_0 \leq a_0$. 

\begin{definition}\label{def:b_0}
	We define
	$$
	b_0 = b_0(P) :=  \max \left \{ \frac{\operatorname{deg} Q_k }{m-k} \, | \, k = 0, \ldots, m-1 \right \}.
	$$
\end{definition}

\begin{remark}\label{rem:relation a_0 and b_0}
	By evaluating  the inequalities from Lemma \ref{lemma-a_0}$(i)$ for $t=0$, we obtain that there are $C,R > 0$ such that for all $k = 1, \ldots, m$
	$$|x|^{a_0k}|Q_k(x)|\leq C |Q_0(x)|, \qquad |x| \geq R.$$
	Hence,
	\begin{equation}\label{eq:Q_0 dominates}
		|P(x,|x|^{a_0})|\leq C|Q_0(x)|, \qquad |x| \geq R,
	\end{equation}
	(with  a different $C$) and
	$$a_0\leq\min\left\{\frac{\operatorname{deg} Q_0-\operatorname{deg}Q_k }{k} \, | \, 1 \leq k \leq m\right \}\leq\frac{\operatorname{deg}Q_0}{m}\leq b_0.$$
	Since $a_0\geq 0$, we also have that $\operatorname{deg}Q_k\leq\operatorname{deg}Q_0$  for all $k = 1, \ldots, m$.
\end{remark}

The main results of this article  are valid under the assumption $a_0 = b_0$. In Section \ref{sect-examples} we calculate $a_0(P), b_0(P), \gamma_0(P)$ and $\mu_0(P)$ for semi-elliptic polynomials $P$.

\subsection{Weight sequences}\label{subsect-weight-sequences}
Let $M = (M_p)_{p \in \N_0}$ be a sequence of positive numbers. We write $M_\alpha = M_{|\alpha|}$ for $\alpha \in \N^d_0$. Furthermore, we set $m_p = M_p/M_{p-1}$ for $p \in \N$. Given $a > 0$, we define $M^a =  (M^a_{p})_{p \in \N_0}$ and $M^{a,*} =  (M^a_{p}/p!)_{p \in \N_0}$. 

We consider the following conditions on a positive sequence $M$:
\begin{itemize}
	\item[$(M.1)$]  $M^2_{p} \leq M_{p-1}M_{p+1}$, $p \in \N$.
	\item[$(M.2)$]   $M_{p+q} \leq CH^{p+q} M_{p} M_{q}$, $p,q\in \N_0$, for some $C,H\geq1$.
	\item[$(M.2)^*$] $2m_p \leq m_{Np}$, $p \geq p_0$, for some $p_0,N \in \N$.
	\item[$(M.3)'$]   $\displaystyle \sum_{p = 0}^\infty \frac{1}{M_p^{1/p}} < \infty$. 
\end{itemize}
We refer to  \cite{Komatsu} for the meaning of the conditions $(M.1)$, $(M.2)$ and $(M.3)'$. Condition $(M.2)^*$ was introduced in \cite{BMM} without a name. Note that $(M.1)$ means that the sequence $(m_p)_{p \in \N}$ is non-decreasing. Hence, any positive sequence $M$ with $M_0= M_1=1$ satisfying $(M.1)$ is non-decreasing. Furthermore, it holds that $M^a$, $a > 0$, satisfies $(M.1)$, $(M.2)$ and $(M.2)^*$, respectively, if $M$ does so. On the other hand, condition $(M.3)'$ is not preserved under taking powers.
\begin{definition}
	A sequence $M = (M_{p})_{p \in \N_0}$ of positive numbers is called a \emph{weight sequence} if
	\begin{itemize}
		\item[$(i)$] $M_0 = M_1 = 1$.
		\item[$(ii)$] $M$ satisfies $(M.1)$, $(M.2)$ and $(M.3)'$. 
	\end{itemize}
\end{definition}
The most important examples of weight sequences  are the \emph{Gevrey sequences} $p!^{\sigma}$ with index $\sigma>1$. 

Given two positive sequences $M$ and $N$, the relation $M \subset N$  means that there are $C,L > 0$ such that
$M_p\leq CL^{p}N_{p}$ for all $p\in\N_0$. The stronger relation $M \prec N$ means that the latter inequality is valid for every $L > 0$ and suitable $C > 0$. We write $M \asymp N$ if both $M \subset N$ and $N \subset M$.

The next condition plays an important role in this article.
\begin{definition}
	Let $a>0$. A positive sequence $M$ is said to satisfy condition $(M.4)_a$ if the sequence $(m^{a,*}_p)_{p \in \N}$ is almost non-decreasing, i.e., $m^{a,*}_p \leq Cm_q^{a,*}$, $p,q \in \N$, $p \leq q$, for some $C > 0$.
\end{definition}
The Gevrey sequence  $p!^{\sigma}$ satisfies $(M.4)_a$ if and only if $\sigma \geq1/a$. In the next lemma we state some consequences of the condition $(M.4)_a$.
\begin{lemma}\label{lemma-M4}
	Let $a > 0$. Let $M$ be a positive sequence satisfying $(M.4)_a$.
	\begin{itemize}
		\item[$(i)$] There exists a positive sequence $N$ with $N \asymp M$ such that $N^{a,*}$ satisfies $(M.1)$.
		\item[$(ii)$] Either $p!^{1/a} \prec M$ or $p!^{1/a} \asymp M$ holds.
		\item[$(iii)$] $M$ satisfies $(M.2)^*$.
	\end{itemize}
\end{lemma}
\begin{proof}
	
	$(i)$ This follows from \cite[Lemma 8]{R-S18}.
	
	$(ii)$ By $(i)$, we may assume without loss of generality that $M^{a,*}$ satisfies $(M.1)$, or equivalently, that $(m^{a,*}_p)_{p \in \N}$ is non-decreasing. If $\lim_{p \to \infty} m^{a,*}_p = \infty$, we have that also $\lim_{p \to \infty} (M^{a,*}_p)^{1/p} = \infty$
	and thus $p!^{1/a} \prec M$. If $\lim_{p \to \infty} m^{a,*}_p < \infty$, it is clear that $p!^{1/a} \asymp M$.
	
	$(iii)$ This follows from \cite[Theorem 3.11]{J-S-S}.
\end{proof}
Let $M$ be a positive sequence with $\lim_{p \to \infty} M_{p}^{1/p} =  \infty$. The \emph{associated function} of $M$ is defined as
$$
\omega_M(\rho) = \sup_{p\in\N_0}\log \frac{\rho^pM_0}{M_p},\qquad \rho \geq 0.
$$
Note that $\log \rho = o(\omega_M(\rho))$ as $\rho \to \infty$.  If $M$ satisfies $(M.1)$ and $(M.2)$, then by \cite[Proposition 3.6]{Komatsu}
$$
2\omega_M(\rho) \leq \omega_M(H\rho) + \log C, \qquad \rho \geq 0,
$$
where $C$ and $H$ are the constants occurring in $(M.2)$. Given two positive sequences $M$ and $N$ satisfying $(M.1)$ with  $\lim_{p \to \infty} M_{p}^{1/p} =  \lim_{p \to \infty} N_{p}^{1/p} =  \infty$, it holds that  $N \subset M$ if and only if
$$
\omega_M(\rho) \leq \omega_N(L\rho) + \log C, \qquad \rho \geq 0,
$$
for some $C, L > 0$ \cite[Lemma 3.8]{Komatsu}. Similarly, $N \prec M$ if and only if the latter inequality remains valid for every $L>0$ and  suitable $C>0$ \cite[Lemma 3.10]{Komatsu}.
We have that $\omega_{p!^\sigma}(\rho) \sim {\rho}^{1/\sigma}$ (which means that $\omega_{p!^\sigma}(\rho) = O({\rho}^{1/\sigma})$ and ${\rho}^{1/\sigma} = O(\omega_{p!^\sigma}({\rho}))$).

\begin{remark}\label{associated-1}  Let $N = (N_{p})_{p \in \N_0}$ be a sequence of positive numbers such that $N \asymp 1$, we define the associated function of $N$ as
	\[\omega_N(\rho) := \begin{cases}
		0,&  0 \leq \rho \leq 1 , \\
		\infty, & \rho > 1.\\
	\end{cases}\]
	Whenever we make use of this function we employ the convention  $0 \cdot \infty  = 0$. This applies in particular to $N=M^{a,*}$, where $M$ is a weight sequence satisfying $p!^{1/a}\asymp M$.
\end{remark}

Finally, we present a lemma that will be used later on.

\begin{lemma}\label{lemma:sequences}
	Let $a > 0$ and let $M$ be a weight sequence.  There are $C,L > 0$ such that
	$$
	M_{\lfloor ap \rfloor} \leq CL^p M^a_p, \qquad p \in \N_0.
	$$
\end{lemma}
\begin{proof}
	Recall that $M^a$  satisfies $(M.1)$ and $(M.2)$ because $M$ does so. Hence,  there are $C, L > 0$ such that
	$$
	a \omega_{M^a}(\rho)  \leq \omega_{M^a}(L\rho) + \log C, \qquad \rho \geq 0.
	$$
	Note that $\omega_{M^a}(\rho) = a \omega_M(\rho^{1/a})$ for $\rho \geq 0$. By \cite[Proposition 3.2]{Komatsu} and the fact that $M \geq 1$, we have that for all $p \in \N_0$
	\begin{align*}
		&M_{\lfloor ap \rfloor} = \sup_{\rho \geq 1} \frac{\rho^{\lfloor ap \rfloor}}{e^{\omega_M(\rho)}} \leq \sup_{\rho \geq 1} \frac{\rho^{ap}}{e^{\omega_M(\rho)}} = \sup_{\rho \geq 1} \frac{\rho^p}{e^{\omega_M(\rho^{1/a})}}  = \sup_{\rho \geq 1} \frac{\rho^{p}}{e^{\frac{1}{a} \omega_{M^a}(\rho)}}  \\
		&\leq C^{1/a} \sup_{\rho \geq 1} \frac{\rho^{p}}{e^{\omega_{M^a}(\rho/L)}} \leq {C^{1/a}} L^p \sup_{\rho > 0} \frac{\rho^{p}}{e^{\omega_{M^a}(\rho)}} ={C^{1/a}}L^p M^a_p.
	\end{align*}
\end{proof}
\subsection{Ultradistributions}\label{subsect: ultradistributions} Fix a weight sequence $M$. Let $K \Subset \R^d$ be such that $\overline{\operatorname{int} K} = K$ and let $h > 0$. We write $\mathscr{E}^{M,h}(K)$ for the Banach space consisting of all $\varphi \in C^\infty(K)$ such that
$$
\| \varphi \|_{\mathscr{E}^{M,h}(K)} := \sup_{\alpha \in \N_0^d}\sup_{x \in K} \frac{|D^\alpha_x\varphi(x)|}{h^{|\alpha|}M_{\alpha}} < \infty. 
$$
For $X \subseteq \R^d$  open we define the spaces of \emph{ultradifferentiable functions of class $(M)$ and $\{M\}$ (of Beurling and Roumieu type) in $X$} as
$$
\mathscr{E}^{(M)}(X) := \varprojlim_{K \Subset X}\varprojlim_{h \rightarrow 0^+} \mathscr{E}^{M,h}(K), \qquad  \mathscr{E}^{\{M\}}(X) := \varprojlim_{K \Subset X} \varinjlim_{h \rightarrow \infty}\mathscr{E}^{M,h}(K).
$$

We employ $[M]$ as a common notation for $(M)$ and $\{M\}$. In addition, we  often first state assertions for the Beurling case followed in parenthesis by the corresponding ones for the Roumieu case.

Let $K \Subset \R^d$ and let $h >0$. We define $\mathscr{D}^{M,h}_K$ as the Banach space consisting of all $\varphi \in C^\infty(\R^d)$ with $\operatorname{supp} \varphi \subseteq K$ such that $\| \varphi \|_{\mathscr{E}^{M,h}(K)} < \infty$. We set
$$
\mathscr{D}^{(M)}_K := \varprojlim_{h \rightarrow 0^+}\mathscr{D}^{M,h}_K, \qquad  \mathscr{D}^{\{M\}}_K :=  \varinjlim_{h \rightarrow \infty}\mathscr{D}^{M,h}_K.
$$
For $X \subseteq \R^d$ open we define 
$$
\mathscr{D}^{[M]}(X) := \varinjlim_{K \Subset X} \mathscr{D}^{[M]}_K. 
$$

Following Komatsu \cite{Komatsu3,Komatsu1991}, we will sometimes make use of an alternative description of the spaces of ultradifferentiable functions of Roumieu type. We denote by $\mathfrak{R}$ the set of all non-decreasing sequences $h = (h_l)_{l \in \N_0}$  of positive numbers such that $h_0 = h_1 = 1$ and $\lim_{l \to \infty} h_l = \infty$. This set is partially ordered and directed by the pointwise order relation $\leq$ on sequences. We remark that we will use $h$ (and $h'$) to denote both positive numbers and elements of $\mathfrak{R}$. In order to avoid confusion, we will always clearly indicate whether $h > 0$ or $h \in \mathfrak{R}$.

Let $K \Subset \R^d$ be such that $\overline{\operatorname{int} K} = K$ and let $h \in \mathfrak{R}$.  We write $\mathscr{E}^{M,h}(K)$ for the Banach space consisting of all $\varphi \in C^\infty(K)$ such that
$$
\| \varphi \|_{\mathscr{E}^{M,h}(K)} := \sup_{\alpha \in \N_0^d}\sup_{x \in K} \frac{|D^\alpha_x\varphi(x)|}{\prod^{|\alpha|}_{l = 0} h_lM_{\alpha}} < \infty.
$$
Let $K \Subset \R^d$ and let $h \in \mathfrak{R}$. We define $\mathscr{D}^{M,h}_K$ as the Banach space consisting of all $\varphi \in C^\infty(\R^d)$ with $\operatorname{supp} \varphi \subseteq K$ such that $\| \varphi \|_{\mathscr{E}^{M,h}(K)} < \infty$. 

\begin{lemma} \label{projective-description}
	Let $M$ be a weight sequence. 
	\begin{itemize}
		\item[$(i)$]\cite[Proposition 3.5]{Komatsu3} Let $X \subseteq \R^d$ be open. Then,
		$$
		\mathscr{E}^{\{M\}}(X) = \varprojlim_{K \Subset X} \varprojlim_{h \in \mathfrak{R}}\mathscr{E}^{M,h}(K)
		$$
		as locally convex spaces.
		\item[$(ii)$] \cite[Proposition 1.1]{Komatsu1991} Let $K \Subset \R^d$. Then,
		$$
		\mathscr{D}^{\{M\}}_K =  \varprojlim_{h \in \mathfrak{R}} \mathscr{D}^{M,h}_K
		$$
		as locally convex spaces.
	\end{itemize}
\end{lemma}
Let $G(z)  = \sum_{\alpha \in \N_0^d} c_\alpha z^\alpha$,  $c_\alpha \in \C$, be an entire function on $\C^d$. The associated infinite order constant coefficient partial differential operator $G(D_x) = \sum_{\alpha \in \N_0^d} c_\alpha D^\alpha_x$
is said to be an \emph{ultradifferential operator of class $[M]$} if for some $h > 0$ (for all $h > 0$)
$$
\sup_{\alpha \in \N_0^d} \frac{|c_\alpha|M_\alpha}{h^{|\alpha|}} < \infty.
$$ 
By \cite[Proposition 4.5]{Komatsu},
$G(D_x)$ is an ultradifferential operator of class $[M]$ if and only if  for some $h >0$ (for all $h > 0$)
$$
\sup_{z \in \C^d} |G(z)| e^{-\omega_M(h|z|)} < \infty.
$$
Condition $(M.2)$ implies that,  for all $X \subseteq \R^d$ open, the linear mappings $G(D_x) :\mathscr{E}^{[M]}(X) \rightarrow \mathscr{E}^{[M]}(X)$ and $G(D_x) :\mathscr{D}^{[M]}(X) \rightarrow \mathscr{D}^{[M]}(X)$ are continuous. 

Let $X \subseteq \R^d$ be open. We denote by $\mathscr{D}'^{[M]}(X)$ and $\mathscr{E}'^{[M]}(X)$ the strong duals of $\mathscr{D}^{[M]}(X)$ and $\mathscr{E}^{[M]}(X)$. The elements of 
$\mathscr{D}'^{[M]}(X)$ are called \emph{ultradistributions of class $[M]$ in $X$}. The space  $\mathscr{E}'^{[M]}(X)$ may be identified with the subspace of $\mathscr{D}'^{[M]}(X)$ consisting of compactly supported elements. Let $G(D_x)$ be an ultradifferential operator of class $[M]$.  We define the mapping  $G(D_x) : \mathscr{D}'^{[M]}(X) \rightarrow \mathscr{D}'^{[M]}(X)$ via transposition. Hence, $G(D_x): \mathscr{D}'^{[M]}(X) \rightarrow \mathscr{D}'^{[M]}(X)$ becomes a continuous linear mapping. 

Later on, we will make use of the following result about the existence of parametrices, which  is essentially shown in  \cite{Gomez-Collado}. It improves a classical result  of Komatsu \cite{Komatsu1991}, namely,  the strong non-quasianalyticity condition $(M.3)$ \cite{Komatsu} is relaxed to the weaker pair of conditions $(M.2)^*$ and $(M.3)'$; the fact that $(M.3)$ implies $(M.2)^*$ follows from \cite[Theorem 3.11]{J-S-S}. Its proof requires the next lemma.

\begin{lemma}\cite[Lemma 2.3]{Prangoski} \label{R-M2} For every $h \in \mathfrak{R}$ there is $h' \in  \mathfrak{R}$ such that $h'_l \leq h_l$ for $l$ large enough and $\prod_{l = 0}^{p+q} h'_l \leq 2^{p+q}  \prod_{l = 0}^{p} h'_l \prod_{l = 0}^{q} h'_l$ for all $p,q \in \N_0$.
\end{lemma}

\begin{lemma}\label{parametrix}
	Let $M$ be a weight sequence satisfying $(M.2)^*$.  Let $h > 0$ ($h \in \mathfrak{R}$) and let $r > 0$. There exist an ultradifferential operator $G(D_x)$ of class $[M]$, $u_1 \in  \mathscr{D}^{M,h}_{\overline{B}(0,r)}$ and $u_2 \in  \mathscr{D}^{[M]}_{\overline{B}(0,r)}$ such that $G(D_x) u_1 = \delta + u_2$ in $\mathscr{D}'^{[M]}(\R^d)$. 
\end{lemma}

\begin{proof}
	We first consider the Beurling case. By \cite[Theorem 14]{BMM}, the function $\omega_M$ is a weight function in the sense of \cite[Definition 1.1]{Gomez-Collado}. Moreover, there {are} $n \in \N$ and $C > 0$ such that \cite[Equation (5)]{BMM}
	$$
	e^{n\varphi^*_{\omega_M}\left( \frac{p}{n} \right)} \leq C(h/2)^p M_p, \qquad p \in \N_0,
	$$
	where $\varphi^*_{\omega_M}$ denotes the Young conjugate of $\varphi_{\omega_M}: [0,\infty) \to [0,\infty), \varphi_{\omega_M}(t) = \omega_M(e^t)$ \cite{Gomez-Collado}.  Hence, \cite[Proposition 2.5]{Gomez-Collado} and conditions $(M.1)$ and $(M.2)$ imply that there are an ultradifferential operator $G(D_x)$ of class $(M)$ and $u \in C^\infty(\R^d)$ such that
	$$
	\sup_{\alpha \in \N_0^d} \sup_{x \in \R^d}  \frac{|D^\alpha_x u(x)|}{(h/2)^{|\alpha|} M_\alpha} < \infty,
	$$
	$u$ is real analytic in $\R^d \backslash \{0 \}$ and $G(D_x) u = \delta$ in $\mathscr{D}'^{(M)}(\R^d)$. Choose $\psi \in    \mathscr{D}^{(M)}_{\overline{B}(0,r)}$ such that $\psi = 1$ in a neighborhood of $0$. Then, $u_1  = \psi u \in   \mathscr{D}^{M,h}_{\overline{B}(0,r)}$ and $u_2 = G(D_x) (\psi u - u) \in  \mathscr{D}^{(M)}_{\overline{B}(0,r)}$ satisfy $G(D_x) u_1 = \delta + u_2$ in $\mathscr{D}'^{(M)}(\R^d)$. Next, we treat the Roumieu case. Note that the positive sequence $N = (\prod_{l=0}^p h_l M_p)_{p \in \N_0}$ satisfies $(M.1)$, $(M.3)'$  and $(M.2)^*$ because $M$ does so. Furthermore, by Lemma \ref{R-M2}, we may assume without loss of generality that $N$ also satisfies $(M.2)$.  Hence, the result can be shown by applying the same reasoning as in the Beurling case to $M = N$ and $h = 1$.
\end{proof}

Finally, we introduce some notation from classical distribution theory. For $K \Subset \R^d$ we denote by $\mathscr{D}_K$ the Fr\'echet space consisting of all smooth functions with support in $K$. For $X \subseteq \R^d$ open we define
$$
\mathscr{D}(X) : = \varinjlim_{K \Subset X} \mathscr{D}_K.
$$
For $l \in \N_0$ we define the norm
$$
\| \varphi \|_{l} = \sup_{x \in \R^d} \max_{|\alpha| \leq l} |D^\alpha_x \varphi(x)|,  \qquad \varphi \in \mathscr{D}(\R^d).
$$
In this context, we sometimes also write $\mathscr{E}(X) = C^\infty(X)$.  We denote by $\mathscr{D}'(X)$ and $\mathscr{E}'(X)$ the strong duals of $\mathscr{D}(X)$ and $\mathscr{E}(X)$, respectively. 
The space  $\mathscr{E}'(X)$ may be identified with the subspace of $\mathscr{D}'(X)$ consisting of compactly supported elements.

\section{A fundamental solution with good regularity and growth properties}\label{sect:FS}
The goal of this auxiliary section is to construct a fundamental solution of $P(D)$ with precise regularity and growth properties.  The next result and its proof are inspired by \cite[p.\ 12--14]{Langenbruch1978}.

\begin{proposition}\label{theo:existence of good fundamental solutions}
	There exists a fundamental solution $E\in\mathscr{D}'(\R^{d+1})$ of $P(D)$ satisfying the following  property: 
	There are $A, L ,S > 0$ such that 
	\begin{equation}\label{improved-regularity} 
		\sup_{(\alpha,l) \in \N_0^{d+1}} \sup_{(x,t) \in \R^{d+1} \backslash \R^d}  \frac{| D^\alpha_x D^l_t  E(x,t)| |t|^{|\alpha|/a_0 + l/\mu_0 + S}e^{-A|t|}} {L^{|\alpha| + l} |\alpha| !^{1/a_0} l!^{1/\mu_0}} < \infty.
	\end{equation}
	In particular, there is $L > 0$ such that for all $l \in \N_0$ there is $S > 0$ such that for all $r > 0$
	\begin{equation}\label{eq:regularity}
		\sup_{\alpha \in \N_0^d} \sup_{\substack{x \in \R^d \\ 0 < |t| \leq r}}  \frac{| D^\alpha_x D^l_t  E(x,t)| |t|^{|\alpha|/a_0 +S}} {L^{|\alpha|} |\alpha| !^{1/a_0}} < \infty.
	\end{equation}		
\end{proposition}

We need some preparation for the proof of Proposition \ref{theo:existence of good fundamental solutions}. Definition \ref{gevrey} and Definition \ref{def:a_0} imply that there are $C,R \geq 1$ such that for all $l \in \N_0$ 
$$
(|x|^{a_0} + |t|^{\mu_0})^l |D^l_t P(x,t)| \leq C |P(x,t)|, \qquad |x| \geq R, t \in \R.
$$
Throughout this section the constants $C$ and $R$ will always refer to those occurring in the above inequality. 
We need the following lemma.

\begin{lemma}\label{lemma:derivatives of reciprocal}
	\begin{itemize}
		\item[$(i)$] There are $C_1,L_1 > 0$ such that for all $l,p\in \N_0$ 
		$$
		\left|D^p_{t} \left( \frac{t^l}{P(x,t)} \right)\right| \leq \frac{C_1L_1^{p+l} p!  (|x|^{a_0} + \max\{1,|t|\})^l}{(|x|^{a_0} + |t|^{\mu_0})^p}, \qquad |x| \geq R,  t \in \R. 
		$$
		\item[$(ii)$] There are $A > 0$ and $C_2,L_2 > 0$ such that for all $l,p \in \N_0$
		$$
		\left|D^p_{t} \left( \frac{(t+iA)^l}{P({x,t+iA})} \right)\right| \leq \frac{C_2L_2^{p+l} p!}{(1 + |t|)^{p-l}}, \qquad |x| \leq R,  t \in \R. 
		$$
	\end{itemize}
\end{lemma}
\begin{proof} Note that $|P(x,t)| \geq 1/C$ for all $ |x| \geq R$ and  $t \in \R$. We first show $(i)$. Set $r = \log( 1+ (2C)^{-1})$. Fix $x \in \R^d$ with $|x| \geq R$ and $t \in \R$ arbitrary. For all $\zeta \in \C$ with $|\zeta|  \leq r (|x|^{a_0} + |t|^{\mu_0})$ it holds that
	$$|P(x, t + \zeta) - P(x,t)| \leq  \sum_{l = 1}^m |D^l_t P(x,t)| \frac{|\zeta|^l}{l!} \leq \frac{1}{2} |P(x,t)| .$$
	Hence, $|P(x, t + \zeta)| \geq 1/(2C)$. 
	Cauchy's inequalities give that for all $l,p \in \N_0$ 
	\begin{align*}
		\left|D^p_{t} \left( \frac{t^l}{P(x,t)} \right)\right| &\leq \frac{p!}{r^p (|x|^{a_0} + |t|^{\mu_0})^p} \max_{|\zeta| \leq r (|x|^{a_0} + |t|^{\mu_0}) } \frac{|t+\zeta|^l }{|P(x,t+\zeta)|} \\
		&\leq \frac{2Cp!(|t| + r(|x|^{a_0}+ |t|^{\mu_0}))^l}{r^p (|x|^{a_0} + |t|^{\mu_0})^p},
	\end{align*}
	from which $(i)$ follows.
	
	Next, we prove $(ii)$. Set $S = \max_{|x| \leq R} \sum_{k = 0}^{m-1} |Q_k(x)|$. For all $x \in \R^d$ with $|x| \leq R$ and  $\zeta \in \C$ with $|\zeta|  \geq \max \{1, S+\frac{1}{2}\}$ it holds that
	$$
	|P(x,\zeta)| \geq |\zeta|^m - \sum_{k = 0}^{m-1} |Q_k(x)| |\zeta|^k \geq \frac{1}{2}.
	$$
	Set $A = 2\max \{1, S+\frac{1}{2}\}$. Fix $x \in \R^d$ with $|x| \leq R$ and $t \in \R$ arbitrary. Cauchy's inequalities give that for all $l,p \in \N_0$ 
	\begin{align*}
		\left|D^p_{t} \left( \frac{(t+iA)^l}{P(x,t+iA)} \right)\right| &\leq \frac{2^pp!}{|t + iA|^p} \max_{|\zeta| \leq |t + iA|/2 } \frac{|t+ iA + \zeta|^l }{|P(x,t+ iA + \zeta)|} \\
		&\leq \frac{2^{p+ l + 1}p!}{|t + iA|^{p-l}}, 
	\end{align*}
	from which $(ii)$ follows.
\end{proof}

\begin{proof}[Proof of Proposition \ref{theo:existence of good fundamental solutions}]
	Set $B = \{ x \in \R^d \, | \, |x| \leq R \}$. With $A$ from Lemma \ref{lemma:derivatives of reciprocal}$(ii)$, we define $E\in\mathscr{D}'(\R^{d+1})$ via
	\begin{align*}
		\langle E, \varphi \rangle =  &\frac{1}{(2\pi)^{d+1}}\int_{\R^d\backslash B}\int_{-\infty}^\infty\frac{\widehat{\varphi}(\xi,\eta)}{P(-\xi,-\eta)}d\eta d\xi\\
		&+\frac{1}{(2\pi)^{d+1}}\int_B \int_{-\infty}^\infty\frac{\widehat{\varphi}(\xi,\eta+iA)}{P(-\xi,-\eta-iA)}d\eta d\xi, 
		\qquad \varphi \in \mathscr{D}(\R^{d+1}),
	\end{align*}
	where $\widehat{\varphi}$ denotes the Fourier transform of $\varphi$. It is clear that $E$ is a fundamental solution of $P(D)$. Since $P$ is hypoelliptic, $E$ is smooth in $\R^{d+1}\backslash \{0\}$. We set 
	$$
	p_{\alpha,l} = \left\lceil  \frac{|\alpha| + d + 1}{{a_0}} + \frac{l+2}{\mu_0}\right \rceil,  \qquad \alpha \in \N_0^d, l \in \N_0.
	$$
	There are $C_3, L_3 > 0$ such that
	\begin{equation}
		p_{\alpha,l} !  \leq C_3 L_3^{|\alpha| + l} |\alpha|!^{1/{a_0}} l!^{1/{\mu_0}}, \qquad \alpha \in \N_0^d, l \in \N_0,
		\label{inequ:factorials}
	\end{equation}
	and $C_4, L_4 > 0$ such that
	\begin{equation}
		\int_{\R^d\backslash B}  |\xi|^{|\alpha|}  \int_{-\infty}^\infty  \frac{(|\xi|^{a_0} + \max\{1,|\eta|\})^l}{(|\xi|^{a_0}+ |\eta|^{\mu_0})^{p_{\alpha,l}}} d\eta d\xi \leq C_4L^l_4, \quad \alpha \in \N_0^d, l \in \N_0.
		\label{inequ:integral}
	\end{equation}
	Let $\psi \in \mathscr{D}(\R^d)$ and $\chi \in \mathscr{D}(\R)$ be arbitrary. Let $r_\chi > 0$ be such that $\supp \chi \subseteq [-r_\chi,r_\chi]$. For all $\alpha\in\N_0^d$ and $l \in \N_0$ it holds that
	\begin{align*}
		&(2\pi)^{d+1}|\langle D_x^\alpha D^l_t  E(x,t) t^{p_{\alpha,l}},\psi(x) \chi(t) \rangle| \\
		&=  \int_{\R^d\backslash B}  |\xi|^{|\alpha|} |\widehat{\psi}(\xi)| \int_{-\infty}^\infty  |\widehat{\chi}(\eta)|   \left | D_t^{p_{\alpha,l}} \left( \frac{\eta^l}{P(-\xi,-\eta)} \right) \right| d\eta d\xi  \\
		&+ \int_B |\xi|^{|\alpha|} |\widehat{\psi}(\xi)|  \int_{-\infty}^\infty  |\widehat{\chi}(\eta + iA)|  \left | D^{p_{\alpha,l}}_t  \left( \frac{(\eta+iA)^l}{P(-\xi,-\eta - iA)}\right) \right| d\eta  d\xi. 
	\end{align*}
	Lemma \ref{lemma:derivatives of reciprocal}$(i)$ and \eqref{inequ:integral} imply that
	\begin{align*}
		& \int_{\R^d\backslash B}  |\xi|^{|\alpha|} |\widehat{\psi}(\xi)| \int_{-\infty}^\infty  |\widehat{\chi}(\eta)|   \left | D_t^{p_{\alpha,l}} \left( \frac{\eta^l}{P(-\xi,-\eta)} \right) \right| d\eta d\xi \\
		&\leq C_1L_1^{p_{\alpha,l} + l} p_{\alpha,l}! \|\psi \|_{L^1} \| \chi\|_{L^1}  \int_{\R^d\backslash B}  |\xi|^{|\alpha|}  \int_{-\infty}^\infty  \frac{(|\xi|^{a_0} + \max\{1,|\eta|\})^l}{(|\xi|^{a_0}+ |\eta|^{\mu_0})^{p_{\alpha,l}}} d\eta d\xi\\ 
		&\leq C_1C_4L_1^{p_{\alpha,l}+l} L_4^l p_{\alpha,l}! \|\psi \|_{L^1} \| \chi\|_{L^1}.   
	\end{align*}
	Likewise, Lemma \ref{lemma:derivatives of reciprocal}$(ii)$ yields that
	\begin{align*}
		&\int_B |\xi|^{|\alpha|} |\widehat{\psi}(\xi)|  \int_{-\infty}^\infty  |\widehat{\chi}(\eta + iA)|  \left | D^{p_{\alpha,l}}_t  \left( \frac{(\eta+iA)^l}{P(-\xi,-\eta - iA)}\right) \right| d\eta  d\xi \\
		&\leq C_2 L_2^{p_{\alpha,l}+l} p_{\alpha,l}! e^{Ar_\chi } \|\psi \|_{L^1} \| \chi\|_{L^1} \int_B |\xi|^{|\alpha|}   \int_{-\infty}^\infty   \frac{1}{(1+|\eta|)^{p_{\alpha,l} - l}} d\eta  d\xi \\
		&\leq 2C_2|B| R^{|\alpha|}L_2^{p_{\alpha,l}+l} p_{\alpha,l}! e^{Ar_\chi} \|\psi \|_{L^1} \| \chi\|_{L^1}. 
	\end{align*}
	By \eqref{inequ:factorials}, there are $C,L > 0$ with
	$$
	|\langle D_x^\alpha D^l_t  E(x,t) t^{p_{\alpha,l}},\psi(x) \chi(t) \rangle| \leq CL^{|\alpha| + l} |\alpha|!^{1/{a_0}} l!^{1/{\mu_0}} e^{Ar_\chi} \|\psi \|_{L^1} \| \chi\|_{L^1}.
	$$
	Fix $x \in \R^d$ and $t \in \R \backslash \{0\}$ arbitrary. Choose $\psi \in \mathscr{D}(\R^d)$ non-negative with $\int_{\R^d} \psi(y) dy = 1$ and  $\chi \in \mathscr{D}(\R)$ non-negative with $\int_{-\infty}^\infty \chi(s) ds = 1$ and $\supp \chi \subseteq [-1,1]$. For $\varepsilon > 0$ we set
	$$
	\psi_{x,\varepsilon}(y) = \frac{1}{\varepsilon^d} \psi \left( \frac{x - y}{\varepsilon} \right), \qquad \chi_{t,\varepsilon}(s) = \frac{1}{\varepsilon} \chi \left( \frac{t - s}{\varepsilon} \right).
	$$
	Then, $\|\psi_{x,\varepsilon}\|_{L^1} = \|\psi \|_{L^1} = 1$, $\|\chi_{t,\varepsilon}\|_{L^1} = \|\chi \|_{L^1} = 1$ and $\supp \chi_{t,\varepsilon} \subseteq [-|t| - \varepsilon,|t| + \varepsilon]$.  We obtain that for all $\alpha \in \N_0^d$ and $l \in \N_0$
	\begin{align*}
		|D_x^\alpha D^l_t  E(x,t) t^{p_{\alpha,l}}| &= \lim_{\varepsilon \to 0^+}|\langle D_x^\alpha D^l_t E(y,s) s^{p_{\alpha,l}},\psi_{x,\varepsilon}(y) \chi_{t,\varepsilon}(s) \rangle| \\
		& \leq CL^{|\alpha| + l} |\alpha|!^{1/{a_0}} l!^{1/{\mu_0}} e^{A|t|},
	\end{align*}
	which implies \eqref{improved-regularity}.
\end{proof}

\section{Characterization of ultradifferentiable functions via (almost) zero solutions of $P(D)$} \label{sect:almost}

In this section we characterize  tuples $(\varphi_0, \ldots, \varphi_{m-1})$ of compactly supported ultradifferentiable functions  via functions $\Phi \in \mathscr{D}(\R^{d+1})$ that are (almost) zero solutions of $P(D)$ and satisfy $D_t^j\Phi( \,\cdot \,,0)=\varphi_j$ for  $j=0,\ldots, m-1$. 
We start with the following result, which is essential for this article.

\begin{proposition}\label{almost-zero}
	Let $M$ be a weight sequence satisfying  $(M.4)_{b_0}$ and let $K \Subset \R^d$.
	There is $L > 0$ such that for  all $h > 0$  the following property holds: For all $\varphi_0, \ldots, \varphi_{m-1} \in \mathscr{D}^{M,h}_K$ there exists $\Phi  = \Phi(\varphi_0, \ldots, \varphi_{m-1}) \in \mathscr{D}(\R^{d+1})$ with $\supp \Phi \subseteq K \times \R$ such that
	\begin{itemize}
		\item[$(i)$] $D^j_t \Phi( \, \cdot \,, 0) = \varphi_j$ for $j = 0, \ldots, m-1$.
		\item[$(ii)$] $\displaystyle \sup_{(x,t) \in \R^{d+1}\backslash \R^d}  |P(D)\Phi(x,t)| e^{\omega_{M^{b_0,*}}\left(\frac{1}{Lh^{b_0}|t|}\right)} < \infty$.
	\end{itemize}
\end{proposition}

\begin{remark}
	By Lemma \ref{lemma-M4}$(ii)$, $(M.4)_{b_0}$ implies that either $p!^{1/b_0} \prec M$ or  $p!^{1/b_0} \asymp M$. If $p!^{1/b_0} \prec M$, the meaning of condition $(ii)$ in Proposition \ref{almost-zero} is clear. If $p!^{1/b_0} \asymp M$, this condition means that $\Phi$ satisfies $P(D) \Phi = 0$ on $\R^d \times (-1/(Lh^{b_0}), 1/(Lh^{b_0}))$ (cf.\ Remark \ref{associated-1}). The same convention will be tacitly used in the rest of this article.
\end{remark}

The proof of Proposition \ref{almost-zero} requires some preparation. In \cite{Petzsche1984} Petzsche constructed almost analytic extensions of ultradifferentiable functions by means of modified Taylor series. If $p!^{1/b_0} \prec M$, we use here a similar idea to prove Proposition \ref{almost-zero}, starting from a power series Ansatz $\Phi$ that formally solves the Cauchy problem $P(D)\Phi = 0$ on $\R^{d+1} \backslash \R^d$ and $D_t^j\Phi( \,\cdot \,,0)=\varphi_j$ on $\R^d$ for $j=0,\ldots, m-1$ \cite[Section 4]{Kalmes18-2} (see also the proof of \cite[Satz 4.1]{Langenbruch1979}). If $p!^{1/b_0} \asymp M$, we even show that $\Phi$  converges. We now recall the definition and some basic properties of this formal power series solution; see  \cite[Section 4]{Kalmes18-2} for  details. For $l \in \N_0$ we recursively define the mapping $\mathscr{C}_l:C^\infty(\R^d)\rightarrow C^\infty(\R^d)$ as
\[
\mathscr{C}_l(\varphi) = \begin{cases}
	0,& l= 0,\ldots,m-2, \\
	\varphi,& l = m-1,\\
	-\sum_{k=0}^{m-1}Q_k(D_x)\mathscr{C}_{k+l-m}(\varphi),& l\geq m.
\end{cases}\]
Since $Q_m=1$, we have that for all $l \in \N_0$
\begin{equation}
	\sum_{k=0}^m Q_k(D_x)\mathscr{C}_{k+l}(\varphi)=0, \qquad  \varphi \in C^\infty(\R^d).
	\label{definition}
\end{equation}
Moreover, as $\mathscr{C}_l$  is a constant coefficient partial differential operator,  $\supp\mathscr{C}_l(\varphi)\subseteq\supp \varphi$ and $D_x^\alpha\mathscr{C}_l(\varphi)=\mathscr{C}_l(D_x^\alpha\varphi)$ for all $\varphi \in C^\infty(\R^d)$ and $\alpha\in\N_0^d$. If $\varphi_0,\ldots,\varphi_{m-1}\in C^\infty(\R^d)$ are such that 
$$
S(\varphi_j)(x,t) = \sum_{l=0}^\infty\mathscr{C}_l(\varphi_j)(x)\frac{(it)^l}{l!}, \qquad (x,t) \in \R^{d+1},
$$ 
converges in $C^\infty(\R^{d+1})$ for all $j = 0, \ldots, m-1$, then
\[\Phi=\sum_{j=0}^{m-1}\sum_{k=0}^{m-1-j}Q_{j+k+1}(D_x)D_t^k S(\varphi_j)\]
solves the Cauchy problem $P(D)\Phi = 0$ on $\R^{d+1} \backslash \R^d$ and $D_t^j\Phi( \,\cdot \,,0)=\varphi_j$ on $\R^d$ for $j=0,\ldots, m-1$ \cite[Proposition 4.4]{Kalmes18-2}.

We need the following  lemma.
\begin{lemma}\label{lemma:taylor-like}
	Let $M$ be a weight sequence and let $K  \Subset \R^d$. There is $L_1 > 0$ such that for all $h > 0$ the following property holds: For all $\alpha \in \N^d$ there is $C > 0$ such that for all $p \in \N_0$ 
	$$
	\| \mathscr{C}_{p}(D^\alpha_x\varphi) \|_0 \leq C(L_1h^{b_0})^{p}  M^{b_0}_p\|\varphi\|_{\mathscr{E}^{M,h}(K)},  \qquad \varphi \in \mathscr{D}^{M,h}_K.
	$$
\end{lemma}
\begin{proof} By $(M.2)$, it suffices to consider the case $\alpha = 0$. For $l \in \N_0$ we define the auxiliary norm
	$$
	|\varphi |_{l} = \sup_{x \in \R^d} \max_{|\alpha| = l} |D^\alpha_x \varphi(x)|,  \qquad \varphi \in \mathscr{D}(\R^d).
	$$
	There is $R \geq 1$ (only depending on $K$) such that for all $l \in \N_0$
	$$
	\| \varphi \|_{l} \leq R^l | \varphi |_{l}, \qquad \varphi \in \mathscr{D}_K.
	$$
	By Lemma \ref{lemma:sequences}, there are $C,L  \geq 1$ such that for all  $l \in \N_0$ and every $\varphi \in \mathscr{D}^{M,h}_K$ 
	$$ |\varphi |_{\lfloor b_0 l\rfloor} \leq h^{\lfloor b_0 l\rfloor} M_{{\lfloor b_0 l\rfloor}}\|\varphi\|_{\mathscr{E}^{M,h}(K)} 
	\leq \frac{C}{\min\{1,h\}} (Lh^{b_0})^l M^{b_0}_l \|\varphi\|_{\mathscr{E}^{M,h}(K)}.$$
	The recursively defined operators $\mathscr{C}_{m-1+l}$, $l  \in \N_0,$ have the following explicit representation \cite[Proposition 4.5]{Kalmes18-2}
	\[\mathscr{C}_{m-1+l}(\varphi)=\sum_{\beta \in\N_0^m, \sigma(\beta)=l}(-1)^{|\beta|}\binom{|\beta|}{\beta_1,\ldots,\beta_m}\prod_{k=1}^m Q_{m-k}^{\beta_k}(D_x)\varphi,\;  \varphi \in C^\infty(\R^d), \]
	where $\sigma(\beta)=\sum_{j=1}^m j\beta_j$. Choose $L_2 \geq 1$ such that for all $k = 0, \ldots, m-1$
	$$
	\| Q_k(D_x) \varphi \|_0 \leq L_2 \| \varphi\|_{\operatorname{deg} Q_k}, \qquad \varphi \in \mathscr{D}(\R^d).
	$$
	Let $h > 0$ and  $\varphi \in  \mathscr{D}^{M,h}_K$ be arbitrary. For all $\beta \in \N_0^m$ with $\sigma(\beta) = l$ it holds that
	\begin{align*}
		&\left \| \prod_{k=1}^m Q_{m-k}^{\beta_k}(D_x)\varphi \right \|_0 \leq L^{|\beta|}_2 \| \varphi \|_{\sum_{k=1}^m \beta_k\operatorname{deg}  Q_{m-k}} \leq  L_2^{l} \| \varphi \|_{\lfloor b_0 l\rfloor} \\
		&\leq  (L_2R^{b_0})^l | \varphi |_{\lfloor b_0 l\rfloor} \leq  \frac{C}{\min\{1,h\}} (L L_2 R^{b_0}h^{b_0})^l M^{b_0}_l \|\varphi\|_{\mathscr{E}^{M,h}(K)}.
	\end{align*}
	Hence,
	\begin{align*}
		&\| \mathscr{C}_{m-1+l}(\varphi)\|_0 \\
		&\leq  \frac{C}{\min\{1,h\}} (L L_2 R^{b_0}h^{b_0})^l M^{b_0}_l \|\varphi\|_{\mathscr{E}^{M,h}(K)} \sum_{\beta \in\N_0^m, |\beta| \leq l}\binom{|\beta|}{\beta_1,\ldots,\beta_m} \\
		& \leq \frac{C}{\min\{1,h\}}  ((m+1)L L_2 R^{b_0}h^{b_0})^l M^{b_0}_l \|\varphi\|_{\mathscr{E}^{M,h}(K)}. \\
		& \leq \frac{C}{\min\{1,h^{1+b_0(m-1)}\}}  ((m+1) LL_2 R^{b_0}h^{b_0})^{m-1+l} M^{b_0}_{m-1+l} \|\varphi\|_{\mathscr{E}^{M,h}(K)}. 
	\end{align*}
	Since $\mathscr{C}_l = 0$ for $l <m-1$, this shows the result.
\end{proof}

\begin{proof}[Proof of Proposition \ref{almost-zero}]  Lemma \ref{lemma-M4}$(ii)$ yields that  either $p!^{1/b_0} \prec M$ or  $p!^{1/b_0} \asymp M$. Suppose first that $p!^{1/b_0} \prec M$. By  Lemma \ref{lemma-M4}$(i)$, we may assume without loss of generality that $m^{b_0,*}_p \nearrow \infty$.  Throughout this proof $C$ will denote a  positive constant that is independent of $\varphi_0, \ldots, \varphi_{m-1}$  but may vary from place to place. Let $H$ be the constant occurring in $(M.2)$ for $M$. Then, $M^{b_0,*}$ satisfies $(M.2)$ with $H^{b_0}$ instead of $H$. Set $A = 8L_1H^{b_0}$, where $L_1$ denotes the constant from Lemma \ref{lemma:taylor-like}, and $L = H^{b_0}A=8L_1H^{2b_0}$. Choose  $\psi \in \mathscr{D}(\R)$ such that $\operatorname{supp} \psi \subseteq [-2,2]$ and $\psi = 1$ on $[-1,1]$. Let now $h>0$ be given, arbitrary but fixed. Define $\lambda_p=Ah^{b_0}m_{p+1}^{b_0,*}$ for $p\in\N_0$. Note that $\lambda_p \nearrow \infty$.  For $\varphi \in \mathscr{D}^{M,h}_K$ we  define
	\begin{equation}
		S(x,t)= S_\psi(\varphi)(x,t) = \sum_{p=0}^\infty  \mathscr{C}_p(\varphi)(x)  \frac{(it)^p}{p!}\psi(\lambda_p t), \qquad (x,t) \in \R^{d+1}.
		\label{formula-S}
	\end{equation}
	Since $\lambda_p \nearrow \infty$, the above series is finite on $\R^d \times\{ t \in \R \, | \, |t| \geq \varepsilon \}$ for each $\varepsilon > 0$. Hence, $S \in C^\infty(\R^d \times (\R \backslash \{0\}))$. We claim that for all $\alpha \in \N_0^d$ and $l \in \N_0$ 
	\begin{equation}
		\limsup_{t \to 0} \sup_{x \in \R^d} |P(D)(D^\alpha_xD^l_tS)(x,t)| e^{\omega_{M^{b_0,*}}\left(\frac{1}{Lh^{b_0}|t|}\right)} < \infty,
		\label{tech-0}
	\end{equation}
	\begin{equation}
		\lim_{t \to 0} D^\alpha_xD^l_t S(x,t) = D^\alpha_x \mathscr{C}_l(\varphi)(x)  \mbox{ uniformly for $x \in \R^d$}.
		\label{tech-1}
	\end{equation}
	Before we prove these claims, let us show how they entail the result. Property \eqref{tech-1} implies that  $S \in \mathscr{D}(\R^{d+1})$ with $\supp S \subseteq K \times \R$ such that 
	\begin{equation}
		D^l_t S(\, \cdot \,,0) = \mathscr{C}_l(\varphi), \qquad l \in \N_0.
		\label{boundary-S}
	\end{equation}
	For $\varphi_0, \ldots, \varphi_{m-1} \in \mathscr{D}^{M,h}_K$ we  define
	$$
	\Phi = \Phi(\varphi_0, \ldots, \varphi_{m-1}) = \sum_{j=0}^{m-1}\sum_{k=0}^{m-1-j}Q_{j+k+1}(D_x)D_t^k S_\psi(\varphi_j).
	$$
	Then, $\Phi   \in \mathscr{D}(\R^{d+1})$ with $\supp \Phi \subseteq K \times \R$. Moreover, \eqref{tech-0} implies that $\Phi$ satisfies $(ii)$. Next, we show $(i)$. In \cite[Proposition 4.2]{Kalmes18-2} it is shown that
	$$
	\sum_{j=0}^n \sum_{k = m-1-n}^{m-1-j} Q_{j+k+1}(D_x) \mathscr{C}_{k+n}(\varphi_j) = \varphi_n, \qquad n= 0, \ldots, m-1.
	$$
	Combining this with \eqref{boundary-S} and the fact that $\mathscr{C}_l = 0$ for $l <m-1$, we obtain that for all $n= 0, \ldots, m-1$
	\begin{align*}
		D^n_t \Phi( \, \cdot \,, 0) &=  \sum_{j=0}^{m-1}\sum_{k=0}^{m-1-j}Q_{j+k+1}(D_x)D_t^{k+n} S_\psi(\varphi_j)(\, \cdot \,,0) \\
		&= \sum_{j=0}^{m-1}\sum_{k=0}^{m-1-j}Q_{j+k+1}(D_x) \mathscr{C}_{k+n}(\varphi_j) \\
		&= \sum_{j=0}^{n}\sum_{k=m-1-n}^{m-1-j}Q_{j+k+1}(D_x) \mathscr{C}_{k+n}(\varphi_j) = \varphi_n.
	\end{align*}
	We now show \eqref{tech-0} and \eqref{tech-1}. Let $\alpha \in \N_0^d$ and $l \in \N_0$ be arbitrary. We have that $D^\alpha_xD^l_t S = S_1 + S_2$, where
	$$
	S_1(x,t) =  \sum_{n =0}^{l-1} \binom{l}{n} \sum_{p=0}^\infty \mathscr{C}_{p+n}(D^\alpha_x\varphi)(x) \frac{(it)^p}{p!} 
	\lambda^{l-n}_{p+n} D^{l-n}_t \psi(\lambda_{p+n}t)
	$$
	and
	$$
	S_2(x,t) = \sum_{p=0}^\infty \mathscr{C}_{p+l}(D^\alpha_x\varphi)(x) \frac{(it)^p}{p!} \psi(\lambda_{p+l}t).
	$$
	We first prove \eqref{tech-0}. Note that
	\begin{align*}
		P(D)S_1(x,t) &=    \sum_{n =0}^{l-1} \binom{l}{n} \sum_{k = 0}^m \sum_{j= 0}^k \binom{k}{j}  \sum_{p=0}^\infty  \mathscr{C}_{p+j+n}\left(Q_k(D_x)D^\alpha_x\varphi\right)(x) \frac{(it)^p}{p!}  \times \\ 
		& \phantom{=} \lambda^{l-n+k-j}_{p+j+n} D^{l-n+k-j}_t \psi(\lambda_{p+j+n}t)
	\end{align*}
	and $P(D)S_2 = T_1 + T_2$, where
	$$
	T_1(x,t) = \sum_{k = 0}^m \sum_{j= 0}^{k-1} \binom{k}{j}  \sum_{p=0}^\infty  \mathscr{C}_{p+j+l}(Q_k(D_x)D^\alpha_x \varphi)(x) \frac{(it)^p}{p!} \lambda^{k-j}_{p+j+l} D^{k-j}_t \psi(\lambda_{p+j+l}t)
	$$
	and
	$$
	T_2(x,t) = \sum_{p=0}^\infty \sum_{k = 0}^m \mathscr{C}_{p+k + l}(Q_k(D_x)D^\alpha_x\varphi)(x) \frac{(it)^p}{p!} \psi(\lambda_{p+k+l}t).
	$$
	We now show suitable estimates for the above expressions. To this end, following Dyn'kin \cite{Dynkin} (see also \cite{Petzsche1984,R-S18}), we introduce the auxiliary function
	$$
	\Gamma(\varepsilon) = \min \left\{ p \in \N_0 \, | \, m^{b_0,*}_{p+1}\geq \frac{1}{\varepsilon} \right\}, \qquad  \varepsilon > 0. 
	$$
	Since $\lim_{p \to \infty} m^{b_0,*}_p = \infty$, it holds that $\Gamma(\varepsilon) < \infty$  for each $\varepsilon > 0$. Fix $\varepsilon > 0$. By definition of $\Gamma$, we have that  $\varepsilon m_p^{b_0,*} < 1$ for all $p \in \N$ with $p \leq \Gamma(\varepsilon)$. Since $(m^{b_0,*}_p)_{p \in \N}$ is non-decreasing, we also have that  $\varepsilon m_p^{b_0,*} \geq 1$ for all $p \in \N$ with $p > \Gamma(\varepsilon)$. As $M^{b_0,*}_0 = 1$, we find that the sequence $p \mapsto \varepsilon^pM^{b_0,*}_p$ is decreasing for $0 \leq p \leq \Gamma(\varepsilon)$ and non-decreasing for $p \geq \Gamma(\varepsilon)$. Consequently,  $\varepsilon^{\Gamma(\varepsilon)} M^{b_0,*}_{\Gamma(\varepsilon)} = e^{-\omega_{M^{b_0,*}}\left (\frac{1}{\varepsilon} \right)}$. We remark that these properties of $\Gamma$, which will be frequently used in the rest of the proof, depend crucially on the assumptions $p!^{1/b_0} \prec M$ and $(M.4)_{b_0}$.  Note that, for all $t \in \R \backslash \{0\}$ and  $p \in \N$ it holds that $\psi(\lambda_pt) = 1$ if $p < \Gamma(A h^{b_0}|t|)$ and  $\psi(\lambda_pt) = 0$ if $p \geq \Gamma(A h^{b_0}|t|/2)$. For all $x \in \R^d$ and $t \in \R \backslash \{ 0\}$ with $|t|$ small enough it thus follows
	\begin{align*}
		P(D)S_1(x,t) &=   \sum_{n =0}^{l-1} \binom{l}{n} \sum_{k = 0}^m \sum_{j= 0}^k \binom{k}{j}  \sum_{ p = \Gamma(A h^{b_0}|t|) -m -l}^{\Gamma(A h^{b_0}|t|/2)-1}  \mathscr{C}_{p+j+n}(Q_k(D_x)D^\alpha_x\varphi)(x)   \times \\ 
		& \phantom{=} \frac{(it)^p}{p!}\lambda^{l-n+k-j}_{p+j+n} D^{l-n+k-j}_t \psi(\lambda_{p+j+n}t),
	\end{align*}
	\begin{align*}
		T_1(x,t) &= \sum_{k = 0}^m \sum_{j= 0}^{k-1} \binom{k}{j}  \sum_{ p = \Gamma(A h^{b_0}|t|) -m -l}^{\Gamma(A h^{b_0}|t|/2)-1}  \mathscr{C}_{p+j+l}( Q_k(D_x)D^\alpha_x\varphi)(x) \frac{(it)^p}{p!} \times \\
		& \phantom{=} \lambda^{k-j}_{p+j+l} D^{k-j}_t \psi(\lambda_{p+j+l}t),
	\end{align*}
	and, by \eqref{definition},
	$$
	T_2(x,t) =  \sum_{k = 0}^m \sum_{p=  \Gamma(A h^{b_0}|t|) -m - l}^{\Gamma(A h^{b_0}|t|/2)-1} \mathscr{C}_{p+k + l}(Q_k(D_x)D^\alpha_x \varphi)(x) \frac{(it)^p}{p!} \psi(\lambda_{p+k+l}t). 
	$$
	In order to estimate the inner sums of  $P(D)S_1$, $T_1$ and $T_2$, let $n \leq l$,  $k \leq m$ and $j \leq k$ be arbitrary. Note that	\begin{equation}
		(p+q)! \leq 2^{p+q}p!q!, \qquad  p,q \in \N_0.
		\label{factorial}
	\end{equation}
	For all $x \in \R^d$ and $t \in \R \backslash \{ 0\}$ with $|t|$ small enough it holds that
	\begin{align*}
		&\left |\sum_{ p = \Gamma(A h^{b_0}|t|) -m -l}^{\Gamma(A h^{b_0}|t|/2)-1}  \mathscr{C}_{p+j+n}(Q_k(D_x)D^\alpha_x \varphi)(x) \frac{(it)^p}{p!}  \lambda^{l-n+k-j}_{p+j+n} D^{l-n+k-j}_t \psi(\lambda_{p+j+n}t) \right | \\
		&\leq C\| \varphi\|_{\mathscr{E}^{M,h}(K)} \sum_{ p = \Gamma(A h^{b_0}|t|) -m -l}^{\Gamma(A h^{b_0}|t|/2)-1}  (2L_1h^{b_0})^p M^{b_0}_{p+j+n} \lambda^{l-n+k-j}_{p+j+n}  \frac{|t|^p}{(p+j+n)!}  \\
		&  \mbox{{(Lemma \ref{lemma:taylor-like} $\&$ inequality \eqref{factorial})}} \\
		&\leq C\| \varphi\|_{\mathscr{E}^{M,h}(K)} \sum_{ p = \Gamma(A h^{b_0}|t|) -m -l}^{\Gamma(A h^{b_0}|t|/2)-1}  (2L_1h^{b_0} |t|)^p M^{b_0,*}_{p+l+k} 
		\\ & \mbox{{($(m^{b_0,*}_p)_{p \in \N}$ is non-decreasing)}}  \\
		&\leq C\| \varphi\|_{\mathscr{E}^{M,h}(K)} \sum_{ p = \Gamma(A h^{b_0}|t|) -m -l}^{\Gamma(A h^{b_0}|t|/2)-1} \frac{1}{2^p} (A h^{b_0} |t|/2)^p M^{b_0,*}_{p} \\&  \mbox{{($(M.2)$ for $M^{b_0,*}$)}}  \\
		&\leq C\| \varphi\|_{\mathscr{E}^{M,h}(K)}  (A h^{b_0} |t|/2)^{\Gamma(A h^{b_0}|t|) -m -l} M^{b_0,*}_{\Gamma(A h^{b_0}|t|) -m -l}   \\&\mbox{{(properties of $\Gamma$)}} \\
		&\leq C\| \varphi\|_{\mathscr{E}^{M,h}(K)} e^{-\omega_{M^{b_0,*}}\left (\frac{1}{A h^{b_0}|t|} \right) + (m+l)\log\left(\frac{1}{A h^{b_0}|t|} \right)} \\&  \mbox{{($M^{b_0,*}$ is non-decreasing $\&$ $\varepsilon^{\Gamma(\varepsilon)}M_{\Gamma(\varepsilon)}^{b_0,*}=e^{-\omega_{M^{b_0,*}}(1/\varepsilon)})$ }} \\
		&\leq C\| \varphi\|_{\mathscr{E}^{M,h}(K)} e^{-\frac{1}{2}\omega_{M^{b_0,*}}\left (\frac{1}{A h^{b_0}|t|} \right)}   \\&\mbox{{($\log(\rho)=o(\omega_{M^{b_0,*}}(\rho))$)}} \\
		&\leq C\| \varphi\|_{\mathscr{E}^{M,h}(K)} e^{-\omega_{M^{b_0,*}}\left (\frac{1}{L h^{b_0}|t|} \right)}  \\&  \mbox{{($(M.2)$ for $M^{b_0,*}$ in terms of its associated weight function)}},
	\end{align*}
	{where we recall that $L = H^{b_0}A$.} This implies \eqref{tech-0}. Finally, we show \eqref{tech-1}. For all $x \in \R^d$ and $t \in \R \backslash \{ 0\}$ with $|t|$ small enough it holds that 
	$$
	S_1(x,t) =  \sum_{n =0}^{l-1} \binom{l}{n} \sum_{p=1}^{\Gamma(A h^{b_0}|t|/2)-1}  \mathscr{C}_{p+n}(D^\alpha_x\varphi)(x) \frac{(it)^p}{p!} 
	\lambda^{l-n}_{p+n} D^{l-n}_t \psi(\lambda_{p+n}t)
	$$
	and
	$$
	S_2(x,t) - D^\alpha_x \mathscr{C}_l(\varphi)(x)  = \sum_{p=1}^{\Gamma(A h^{b_0}|t|/2)-1} \mathscr{C}_{p+l}(D^\alpha_x\varphi)(x) \frac{(it)^p}{p!} \psi(\lambda_{p+l}t).
	$$
	Let $n \leq l$ be arbitrary. {By using similar arguments as above, we find that} for all $x \in \R^d$ and $t \in \R \backslash \{ 0\}$ with $|t|$ small enough 
	\begin{align*}
		&\left | \sum_{p=1}^{\Gamma(A h^{b_0}|t|/2)-1}  \mathscr{C}_{p+n}(D^\alpha_x\varphi)(x) \frac{(it)^p}{p!} 
		\lambda^{l-n}_{p+n} D^{l-n}_t \psi(\lambda_{p+n}t)  \right | \\
		\leq&  C \| \varphi\|_{\mathscr{E}^{M,h}(K)} \sum_{p=1}^{\Gamma(A h^{b_0}|t|/2)-1} (2L_1h^{b_0})^p  M^{b_0}_{p+n} \lambda^{l-n}_{p+n} \frac{|t|^p}{(p+n)!} \\
		\leq&  C\| \varphi\|_{\mathscr{E}^{M,h}(K)} \sum_{p=1}^{\Gamma(A h^{b_0}|t|/2)-1}  (2L_1h^{b_0})^p  M^{b_0,*}_{p+l} |t|^p \\
		\leq& C\| \varphi\|_{\mathscr{E}^{M,h}(K)} \sum_{p=1}^{\Gamma(A h^{b_0}|t|/2)-1}  \frac{1}{2^p} (A h^{b_0} |t|/2)^p M^{b_0,*}_{p}  \\
		\leq& |t| C\| \varphi\|_{\mathscr{E}^{M,h}(K)}, 
	\end{align*}
	which implies \eqref{tech-1} and finishes the proof of the claim.
	
	Next, we suppose that $p!^{1/b_0} \asymp M$. For $\varphi \in \mathscr{D}^{M,h}_K$ we  define
	$$
	S(x,t) = \sum_{p=0}^\infty  \mathscr{C}_p(\varphi)(x)  \frac{(it)^p}{p!}, \qquad (x,t) \in \R^{d+1}.
	$$
	Lemma \ref{lemma:taylor-like} implies that there is $L > 0$ (independent of $h > 0$) such that the series $S(x,t)$ converges in $C^\infty(\R^d \times (-2/(Lh^{b_0}), 2/(Lh^{b_0})))$ and  that for all $\alpha \in \N_0^{d}$ and $l \in \N_0$ there is $C > 0$ such that 
	$$
	\sup_{\substack{ x \in \R^d \\ |t| \leq 2/(Lh^{b_0})  } } |D^\alpha_x D^l_t S(\varphi)(x,t)| \leq C \| \varphi\|_{\mathscr{E}^{M,h}(K)}, \qquad \varphi \in \mathscr{D}^{M,h}_K.
	$$
	Furthermore, $S$ satisfies \eqref{boundary-S} and $P(D)S = 0$ on $\R^d \times (-2/(Lh^{b_0}), 2/(Lh^{b_0}))$ by \eqref{definition}. Choose $\chi \in \mathscr{D}(\R)$ with $\supp\chi\subset (-2/(Lh^{b_0}), 2/(Lh^{b_0}))$ and such that $\chi = 1$ on $[-1/(Lh^{b_0}), 1/(Lh^{b_0})]$. For $\varphi_0, \ldots, \varphi_{m-1} \in \mathscr{D}^{M,h}_K$ we define 
	$$
	\Phi(x,t) = \Phi(\varphi_0, \ldots, \varphi_{m-1})(x,t) = \chi(t) \sum_{j=0}^{m-1}\sum_{k=0}^{m-1-j}Q_{j+k+1}(D_x)D_t^k S(\varphi_j)(x,t).
	$$ 
	Then, $\Phi \in \mathscr{D}(\R^{d+1})$ with $\supp \Phi \subset K \times \R$. Moreover, $P(D)\Phi = 0$ on $\R^d \times (-1/(Lh^{b_0}), 1/(Lh^{b_0}))$ and  the exact same argument as in the first part of the proof  shows that $\Phi$ satisfies $(i)$.
\end{proof}

\begin{remark}\label{remark-continuity}
	Choose $\chi \in \mathscr{D}(\R)$ such that $\chi = 1$ on a neighborhood of $0$. Let $\Phi$ be the function constructed in Proposition \ref{almost-zero}. Set
	$$
	||| \Phi ||| = \sup_{(x,t) \in \R^{d+1} \backslash \R^d}  |P(D) (\chi(t)\Phi(x,t))| e^{\omega_{M^{b_0,*}}\left(\frac{1}{Lh^{b_0}|t|}\right)} < \infty.
	$$
	An inspection of the proof of Proposition \ref{almost-zero} shows that  there is $C > 0$ with
	$$
	||| \Phi(\varphi_0, \ldots, \varphi_{m-1}) ||| \leq C\max_{j = 0, \ldots, m-1} \| \varphi_j\|_{\mathscr{E}^{M,h}(K)}, \qquad   \varphi_0, \ldots, \varphi_{m-1} \in \mathscr{D}^{M,h}_K.
	$$
\end{remark}

\begin{proposition}\label{converse-almost-zero}
	Let $M$ be a weight sequence satisfying $(M.4)_{a_0}$. There is $L>0$ such that for all $h> 0$ the following property holds: Let $\Phi  \in \mathscr{D}(\R^{d+1})$ be such that 
	\begin{equation}
		\sup_{(x,t) \in \R^{d+1} \backslash \R^d}  |P(D)\Phi(x,t)| e^{\omega_{M^{a_0,*}}\left(\frac{1}{h|t|}\right)} < \infty.
		\label{assumption-reverse} 
	\end{equation}
	Set $K = \supp \Phi( \,\cdot \,,0)$. Then, $D_t^l\Phi( \,\cdot \,,0) \in \mathscr{D}^{M,Lh^{1/a_0}}_K$ for all $l \in \N_0$.
\end{proposition}

\begin{remark}
	We would like to point out that Proposition \ref{converse-almost-zero} is not needed to prove the main results of this article in Section \ref{sect-bv} and Section \ref{sec:Main results} below. However, we believe this result is interesting in its own right as it provides a complete characterization of ultradifferentiable classes in terms of (almost) zero solutions of $P(D)$ (provided that $a_0 = b_0$); see Theorem \ref{main-thm-almost} and Remark \ref{main-thm-almost-remark} below.
\end{remark}

We need the following lemma to prove Proposition \ref{converse-almost-zero}. 
\begin{lemma}\label{lemma:cons-taylor}
	Let  $K \Subset \R^d$, $r > 0$ and $V \subseteq \R^{d+1}$ open {be} such that $K \times [-r,r] \Subset V$.
	Let $T \in \mathscr{D}'(V)$ be such that $T_{| V \backslash V \cap \R^d} \in L^1_{\operatorname{loc}}(V \backslash V \cap \R^d)$ and suppose that there is $N \in \N$ with
	\begin{equation}
		\sup_{(x,t) \in K \times [-r,r] \backslash K} |T(x,t)||t|^N < \infty.
		\label{bounds-dist}
	\end{equation}
	Furthermore, let $l \in \N_0$ and $C >0$ be such that
	\begin{equation}
		|\langle T, \varphi \rangle|  \leq C \| \varphi \|_l, \qquad \varphi \in \mathscr{D}_{K \times [-r,r]}.
		\label{order}
	\end{equation}
	Let $\varphi \in  \mathscr{D}_{K \times [-r,r]}$ be such that
	\begin{equation}
		\sup_{(x,t) \in \R^{d+1}\backslash \R^d} \frac{|\varphi(x,t)|}{|t|^{\max\{N,l+1\}}} < \infty
		\label{bounds-test}
	\end{equation}
	Then,
	$$
	\langle T, \varphi \rangle = \int _\R\int_{\R^d} T(x,t) \varphi(x,t) dx dt.
	$$
\end{lemma}

\begin{proof} In view of \eqref{bounds-test}, Taylor's theorem yields that $D^n_t\varphi( \, \cdot \,, 0) = 0$ for all $n \leq l$. Hence, also $D^\alpha_xD^n_t\varphi( \, \cdot \,, 0) = 0$ for all $\alpha \in \N^d$ and $n \leq l$. By another application of Taylor's theorem, we obtain that
	\begin{equation}
		\sup_{(x,t) \in \R^{d+1}\backslash \R^d} \frac{|D^\alpha_x D^{n-k}_t\varphi(x,t)|}{|t|^{k+1}} < \infty, \qquad \alpha \in \N^d, k \leq n \leq l.
		\label{taylor-forever}
	\end{equation}
	Choose $\psi \in \mathscr{D}(\R)$ such that $\psi = 1$ on a neighborhood of $0$. Set $\psi_\varepsilon(t) = \psi(t/\varepsilon)$ for $\varepsilon > 0$. Property \eqref{taylor-forever} implies that
	$$
	\lim_{\varepsilon \to 0^+} \|\psi_\varepsilon(t) \varphi(x,t) \|_l =  0.
	$$
	Hence, by \eqref{order},
	\begin{align*}
		\langle T, \varphi \rangle &= \lim_{\varepsilon \to 0^+} \langle T(x,t), (1-\psi_\varepsilon(t))\varphi(x,t) \rangle \\
		&= \lim_{\varepsilon \to 0^+}  \int_\R \int_{\R^d} T(x,t)  (1-\psi_\varepsilon(t))\varphi(x,t) dx dt \\
		&=  \int_\R \int_{\R^d} T(x,t) \varphi(x,t) dx dt,
	\end{align*}
	where the last equality is justified by Lebesgue's dominated convergence theorem and the inequalities \eqref{bounds-dist} and \eqref{bounds-test}.
\end{proof}

\begin{proof}[Proof of Proposition \ref{converse-almost-zero}]
	By Lemma \ref{lemma-M4}$(ii)$, we have that  either $p!^{1/a_0} \prec M$ or  $p!^{1/a_0} \asymp M$. We only consider the case $p!^{1/a_0} \prec M$ as the case $p!^{1/a_0} \asymp M$ can be treated similarly. Let $E$ be the fundamental solution of $P(D)$ constructed in Proposition \ref{theo:existence of good fundamental solutions}. Let $l \in \N_0$ be arbitrary.  Choose $r > 0$ such that $\supp \Phi \subset \R^d \times [-r,r]$. There are $C,S > 0$ such that\begin{equation}
		\sup_{\substack{x \in \R^d \\ 0 < |t| \leq r}}  | D^\alpha_x D^l_t  E(x,t)| |t|^{|\alpha|/a_0 + S}  \leq C{L^{|\alpha|} |\alpha| !^{1/a_0}}, \qquad \alpha \in \N^d_0.
		\label{good-sol}
	\end{equation}
	Set 
	$$
	p_\alpha = \left \lceil \frac{|\alpha|}{a_0} + S \right \rceil , \qquad \alpha \in \N_0^d.
	$$
	There are $C_1, L_1 > 0$ such that
	$$
	|\alpha|!^{1/a_0} \leq C_1L_1^{|\alpha|}p_\alpha!, \qquad \alpha \in \N_0^d.
	$$
	Lemma \ref{lemma:sequences} and $(M.2)$ imply that there are  $C_2,L_2 > 0$ such that
	$$
	M_{p_\alpha} \leq C_2L_2^{|\alpha|}M^{1/a_0}_{\alpha}, \qquad  \alpha \in \N_0^d.
	$$
	The above two inequalities and \eqref{assumption-reverse} yield that there are $C_3,L_3 > 0$ (with $L_3$ independent of $l$) such that
	\begin{equation}
		\sup_{(x,t) \in \R^{d+1} \backslash \R^d}   \frac{|P(D)\Phi(x,t)|}{|t|^{|\alpha|/a_0 + S}}  \leq \frac{C_3(L_3h^{1/a_0})^{|\alpha|}M_\alpha}{|\alpha|!^{1/a_0}}, \qquad \alpha \in \N^d.
		\label{assumption-mod}
	\end{equation} 
	Lemma \ref{lemma:cons-taylor} implies that for all  $x \in \R^d$ and $\alpha \in \N_0^d$
	\begin{align*}
		D^\alpha_xD^l_t \Phi(x,0) &= (D^\alpha_x D^l_t E \ast P(D) \Phi)(x,0) \\
		&= \langle  D^\alpha_xD^l_tE(x-y, -t), P(D)\Phi (y,t) \rangle \\
		&=   \int_\R \int_{\R^d} D^\alpha_xD^l_tE(x-y, -t) P(D)\Phi (y,t) dy dt.
	\end{align*}
	The inequalities \eqref{good-sol} and \eqref{assumption-mod} imply that for  all $\alpha \in \N_0^d$
	\begin{align*}
		\sup_{x \in \R^d} |D^\alpha_xD^l_t \Phi(x,0)| &\leq \sup_{x \in \R^d} \int_\R \int_{\R^d} |D^\alpha_xD^l_tE(x-y, -t)| | P(D)\Phi (y,t) | dy dt \\
		&\leq  |\supp \Phi |CC_3 (LL_3h^{1/a_0})^{|\alpha|} M_\alpha.
	\end{align*}
\end{proof}

Proposition \ref{almost-zero} and Proposition \ref{converse-almost-zero} yield the following  result. 

\begin{theorem}\label{main-thm-almost}
	Suppose that $a_0 = b_0$. Let $M$ be a weight sequence satisfying $(M.4)_{a_0}$ (and $p!^{1/a_0} \prec M$ in the Beurling case). Let $X \subseteq \R^d$ be open and let $V \subseteq \R^{d+1}$ be open such that $V \cap \R^d = X$. Let $\varphi_0, \ldots, \varphi_{m-1} \in \mathscr{D}(X)$. Then,  $\varphi_0, \ldots, \varphi_{m-1} \in \mathscr{D}^{[M]}(X)$ if and only if for all $h > 0$ (for some $h>0$) there exists $\Phi  \in \mathscr{D}(V)$ such that
	\begin{itemize}
		\item[$(i)$] $D^j_t \Phi( \, \cdot \,, 0) = \varphi_j$ for $j = 0, \ldots, m-1$.
		\item[$(ii)$] $\displaystyle \sup_{(x,t) \in \R^{d+1} \backslash \R^d}    |P(D)\Phi(x,t)| e^{\omega_{M^{a_0,*}}\left(\frac{1}{h|t|}\right)} < \infty$.
	\end{itemize}
\end{theorem}

\begin{remark}\label{main-thm-almost-remark}
	For $M \asymp p!^{1/a_0}$  the following analogue of Theorem \ref{main-thm-almost} holds in the Beurling case: \emph{Suppose that $a_0 = b_0$. Let $X \subseteq \R^d$ be open. Let $\varphi_0, \ldots, \varphi_{m-1} \in \mathscr{D}(X)$. Then,  $\varphi_0, \ldots, \varphi_{m-1} \in \mathscr{D}^{(M)}(X)$ if and only if for all $h > 0$ there exists $\Phi  \in \mathscr{D}(X \times \R)$ such that
		\begin{itemize}
			\item[$(i)$] $D^j_t \Phi( \, \cdot \,, 0) = \varphi_j$ for $j = 0, \ldots, m-1$.
			\item[$(ii)$] $P(D)\Phi(x,t) = 0$ on $X \times (-1/h, 1/h)$.
		\end{itemize}
	}
\end{remark}

We end this section by giving an analogue of Proposition \ref{almost-zero} and Remark \ref{remark-continuity} for compactly supported smooth functions.

\begin{proposition}\label{almost-zero-smooth}  Let $K \Subset \R^d$ and let $\varepsilon > 0$. Let $N \in \N_0$. For all $\varphi_0, \ldots, \varphi_{m-1} \in \mathscr{D}_K$ there exists $\Phi  = \Phi(\varphi_0, \ldots, \varphi_{m-1}) \in \mathscr{D}_{K \times [-\varepsilon, \varepsilon]}$ such that
	\begin{itemize}
		\item[$(i)$] $D^j_t \Phi( \, \cdot \,, 0) = \varphi_j$ for $j = 0, \ldots, m-1$.
		\item[$(ii)$] $||| \Phi ||| = \displaystyle \sup_{(x,t) \in \R^{d+1} \backslash \R^d}   \frac{|P(D)\Phi(x,t)|}{|t|^N}  < \infty$.
	\end{itemize}
	Furthermore, there are $l \in \N_0$ and $C > 0$ such that
	$$
	||| \Phi(\varphi_0, \ldots, \varphi_{m-1}) ||| \leq C\max_{j = 0, \ldots, m-1} \| \varphi_j\|_{l}, \qquad   \varphi_0, \ldots, \varphi_{m-1} \in \mathscr{D}_K.
	$$
\end{proposition}

\begin{proof} Let $\psi \in \mathscr{D}_{[-\varepsilon,\varepsilon]}$ be such that $\psi = 1$ on a neighborhood of $0$.  For $n \in \N$ and $\varphi \in \mathscr{D}_K$ we define
	$$
	S_n(x,t)  = S_{n}(\varphi)(x,t) = \psi(t)\sum_{p=0}^n  \mathscr{C}_p(\varphi)(x)  \frac{(it)^p}{p!}, \qquad (x,t) \in \R^{d+1}.
	$$
	The result can now be shown in a similar way as Proposition \ref{almost-zero} but starting from $S= S_{n}$ with $n$ large enough instead of the function defined in \eqref{formula-S}.
\end{proof}

\section{Boundary values of zero solutions of $P(D)$}\label{sect-bv}

In this section we show that zero solutions of $P(D)$ satisfying suitable growth estimates near $\R^d$ have  boundary values in a given ultradistribution space. We start with the following fundamental definition.
\begin{definition}
	Let $M$ be a weight sequence. Let $X \subseteq \R^d$ be open and let $V \subseteq \R^{d+1}$ be open such that $V \cap \R^d = X$. The \emph{boundary value} $\bv(f) \in \mathscr{D}'^{[M]}(X)$ of an element $f \in C^\infty_P(V \backslash X)$ is defined as
	\begin{equation}
		\langle \bv(f), \varphi \rangle := \lim_{t,s \to 0^+} \int_{X} (f(x,t ) - f(x,-s))  \varphi(x) dx, \qquad \varphi \in \mathscr{D}^{[M]}(X),
		\label{def-bv}
	\end{equation}
	provided that  $\bv(f)  \in \mathscr{D}'^{[M]}(X)$ exists. 
\end{definition}
Since $\mathscr{D}^{[M]}(X)$ is barrelled,  $\bv(f)  \in \mathscr{D}'^{[M]}(X)$ exists if and only if the limit in the right-hand side of \eqref{def-bv} exists and is finite for all $\varphi \in \mathscr{D}^{[M]}(X)$.
\begin{lemma}\label{x-derivatives}
	Let $M$ be a weight sequence. Let $X \subseteq \R^d$ be open and let $V \subseteq \R^{d+1}$ be open such that $V \cap \R^d = X$. Let $f \in C^\infty_P(V \backslash X)$ be such that $\bv(f)  \in \mathscr{D}'^{[M]}(X)$ exists. Then, 
	\begin{itemize}
		\item[$(i)$] $\bv(D^\alpha_xf) = D^\alpha_x \bv(f)   \in \mathscr{D}'^{[M]}(X)$ exists for all $\alpha \in \N_0^d$. 
		\item[$(ii)$] Assume additionally that $p!^{1/\gamma_0} \prec M$ ($p!^{1/\gamma_0} \subset M$). Let $G(D_x)$  be an ultradifferential operator of class $[M]$. Then, $G(D_x)f \in  C^\infty_P(V \backslash X)$ and  $\bv(G(D_x)f) = G(D_x) \bv(f) \in \mathscr{D}'^{[M]}(X)$ exists.
	\end{itemize}
\end{lemma}
\begin{proof}
	$(i)$ Obvious.
	
	$(ii)$ The condition $p!^{1/\gamma_0} \subset M$ (and thus also $p!^{1/\gamma_0} \prec M$) implies that $G(D_x)$ is an ultradifferential operator of class $\{p!^{1/\gamma_0}\}$. Since $f \in C^\infty_P(V \backslash X) \subset \Gamma^{1/\gamma_0, 1/\mu_0}(V\backslash X)$ (see Subsection \ref{subsect:Diff-oper}), we have that $G(D_x)f \in C^\infty(V\backslash X)$ and  $P(D) G(D_x) f = G(D_x) P(D) f = 0$. Hence, $G(D_x)f  \in C^\infty_P(V \backslash X)$. The second statement is clear.
\end{proof}

We now  introduce various weighted spaces of  zero solutions of $P(D)$. Let $M$ be a positive sequence. Suppose that $p! \prec M$. Let $X \subseteq \R^d$ be open and let $V \subseteq \R^{d+1}$ be open such that $V \cap \R^d = X$.  For $h > 0$ we define $\mathscr{B}_{P,M,h}(V \backslash X)$ as the Banach space consisting of all $f \in C^\infty_P(V \backslash X)$ such that
$$
\| f \|_{\mathscr{B}_{P,M,h}(V \backslash X)} := \sup_{(x,t) \in V \backslash X} |f(x,t)| e^{-\omega_{M^{1,*}}\left(\frac{1}{h|t|}\right)} < \infty.
$$
Let $b > 0$ and suppose that $p!^{1/b} \prec M$. Choose a sequence $(V_l)_{l \in \N_0}$ of relatively compact open subsets of $V$ such that $\overline{V}_l \Subset V_{l+1}$ and $V = \bigcup_{l \in \N_0} V_{l}$. Set $X_l = V_l \cap X$ for $l \in \N_0$. We define
\begin{gather*}
	C^\infty_{P,(M),b}(V \backslash X) := \varprojlim_{l \in \N_0} \varinjlim_{h \to 0^+} \mathscr{B}_{P,M^{b},h}(V_{l}\backslash X_l), \\
	C^\infty_{P,\{M\},b}(V \backslash X) := \varprojlim_{l \in \N_0} \varprojlim_{h \to \infty} \mathscr{B}_{P,M^{b},h}(V_{l}\backslash X_l).
\end{gather*}
This definition is  independent of the chosen sequence  $(V_{l})_{l \in \N_0}$. Obviously, it holds that $C^\infty_P(V)\subset C^\infty_{P,[M],b}(V\backslash X)$. If $p!^{1/b} \asymp M$, we set $C^\infty_{P,\{M\},b}(V \backslash X) := C^\infty_P(V \backslash X)$. 

Next, we define for  $j = 1, \ldots, m$
$$
P_{(j)}(x,t) = \sum_{k=j}^m Q_k(x) t^{k-j}.
$$
Since $Q_m = 1$, we have the recursion relations 
\begin{equation}
	\label{base-change}
	t^j= \begin{cases}
		P_{(m)},&  j=0, \\
		P_{(m-j)} -  \sum_{k=0}^{j-1}Q_{k+m-j}t^k, &  j = 1, \ldots, m-1.\\
	\end{cases}
\end{equation}

We are ready to discuss the existence of ultradistributional boundary values of zero solutions of $P(D)$.
\begin{proposition}\label{existence-bv}
	Let $M$ be a weight sequence satisfying $(M.4)_{b_0}$ (and $p!^{1/b_0} \prec M$ in the Beurling case). Let $X \subseteq \R^d$ be open and let $V \subseteq \R^{d+1}$ be open such that $V \cap \R^d = X$. For all $j = 0, \ldots, m-1$ the mapping
	$$
	C^\infty_{P,[M],b_0}(V \backslash X) \rightarrow \mathscr{D}'^{[M]}(X), \, f \mapsto \operatorname{bv}(P_{(j+1)}(D) f)
	$$
	is well-defined and continuous.
\end{proposition}
We shall show Proposition \ref{existence-bv} by combining the description of ultradifferentiable functions via (almost) zero solutions of $P(D)$ (Proposition \ref{almost-zero}) with the identity obtained in the next lemma. As mentioned in the introduction, this is a generalization of the method of establishing the existence of  (ultra)distributional boundary values of holomorphic functions  by combining almost analytic extensions with Stokes' theorem (more precisely, the formula in \cite[Equation (3.1.9), p.\ 62]{HoermanderPDO1}).
\begin{lemma} \label{lemma-Green}
	Let $X \subseteq \R^d$ be open and let $a,b \in \R$ with $a < b$.  For all $f\in C^\infty(X \times [a,b])$ and $\Phi\in\mathscr{D}(X \times \R)$ it holds that
	\begin{align*}
		&\int_a^b\int_{X} f(x,s)\check{P}(D)\Phi(x,s)dxds =\int_a^b\int_{X}P(D)f(x,s)\,\Phi(x,s)dxds\\
		&+ i\sum_{j=0}^{m-1} (-1)^j  \int_{X} P_{(j+1)}(D)f(x,b)D_t^j\Phi(x,b) - P_{(j+1)}(D)f(x,a)D_t^j\Phi(x,a)dx,
	\end{align*}
	where, as usual, $\check{P}(\xi)=P(-\xi)$.
\end{lemma}
\begin{proof} This follows by multiple partial integrations.
	\end{proof}

\begin{proof}[Proof of Proposition \ref{existence-bv}]
	We only show the Beurling case as the Roumieu case can be treated similarly. Let $ 0 \leq j \leq m-1$ be arbitrary.  It suffices to show that for all relatively compact open subsets $Y$ of $\R^d$, $r,h >0$ and  $K \Subset Y$ the mapping
	$$
	\mathscr{B}_{P,M^{b_0},h}(Y \times (-r,r) \backslash Y) \rightarrow (\mathscr{D}^{(M)}_K)'_b, \, f \mapsto \operatorname{bv}(P_{(j+1)}(D) f)
	$$
	is well-defined and continuous. Let $f \in \mathscr{B}_{P,M^{b_0},h}(Y \times (-r,r) \backslash Y)$ be arbitrary.  We set $f_t(x,s) = f(x,s+t)$ for $0 <t < r/2$. Then, $f_t \in C^\infty(Y \times [0,r/2])$. For $\varphi \in \mathscr{D}^{(M)}_K$ we consider the function $\Phi = \Phi(0, \ldots, \varphi_j = \varphi, \ldots, 0)$ from Proposition \ref{almost-zero} (with $\check{P}$ and   $(h/L)^{1/b_0}$ instead of $P$ and $h$). Choose $\psi \in \mathscr{D}(\R)$ with $\supp \psi  \subseteq [-r/2, r/2]$ and $\psi = 1$ on a neighborhood of $0$. By applying Lemma \ref{lemma-Green} to $f_t$ and $\psi\Phi$,  we obtain that  for $0<t<r/2$ 
	$$
	\int_{Y} P_{(j+1)}(D) f(x,t) \varphi(x) dx = (-1)^ji \int_0^{r/2} \int_Y f(x,s+t) \check{P}(D) (\psi(s)\Phi(x,s)) dxds.
	$$
	Likewise, one may show that 
	\begin{align*}
		&\int_{Y} P_{(j+1)}(D) f(x,-t) \varphi(x) dx\\ =& (-1)^{j+1}i \int_{-r/2}^0 \int_Y f(x,s-t) \check{P}(D) (\psi(s)\Phi(x,s)) dxds.
	\end{align*}
	Hence,
	$$
	\langle \operatorname{bv}(P_{(j+1)}(D) f), \varphi \rangle = (-1)^ji \int_{-r/2}^{r/2} \int_Y f(x,s) \check{P}(D)(\psi(s) \Phi(x,s)) dxds, 
	$$
	where the integral at the right hand-side is convergent by property $(ii)$ of Proposition \ref{almost-zero}. The result now follows from Remark \ref{remark-continuity}.
\end{proof}

\begin{proposition}\label{existence-bv-1}
	Let $M$ be a weight sequence satisfying $(M.4)_{b_0}$ (and $p!^{1/b_0} \prec M$ in the Beurling case). Let $X \subseteq \R^d$ be open and let $V \subseteq \R^{d+1}$ be open such that $V \cap \R^d = X$. For all $l \in \N_0$ the mapping
	$$
	C^\infty_{P,[M],b_0}(V\backslash X) \rightarrow \mathscr{D}'^{[M]}(X), \, f \mapsto \operatorname{bv}(D^l_t f)
	$$
	is well-defined and continuous. 
\end{proposition}
\begin{proof}
	For $l = 0, \ldots, m-1$ this is a consequence  of  \eqref{base-change}, Lemma \ref{x-derivatives} and Proposition \ref{existence-bv}.  As $Q_m = 1$, the result for $l \geq m$ then follows by using Lemma \ref{x-derivatives} recursively.
\end{proof}

Finally, we discuss the existence of distributional boundary values of zero solutions of $P(D)$. 
\begin{definition}
	Let $X \subseteq \R^d$ be open and let $V \subseteq \R^{d+1}$ be open such that $V \cap \R^d = X$.  The boundary value $\bv(f) \in \mathscr{D}'(X)$ of an element $f \in C^\infty_P(V \backslash X)$ is defined as
	$$
	\langle \bv(f), \varphi \rangle := \lim_{t,s \to 0^+} \int_{X} (f(x,t ) - f(x,-s))  \varphi(x) dx, \qquad \varphi \in \mathscr{D}(X),
	$$
	provided that  $\bv(f)  \in \mathscr{D}'(X)$ exists. 
\end{definition}
Let $V \subseteq \R^{d+1}$ be open such that $V \cap \R^d = X$. For $N \in \N_0$ we define $\mathscr{B}_{P,N}(V \backslash X)$ as the Banach space consisting of all $f \in C^\infty_P(V \backslash X)$ such that
$$
\| f \|_{\mathscr{B}_{P,N}(V \backslash X)} := \sup_{(x,t) \in V \backslash X} |f(x,t)| |t|^N< \infty.
$$
Choose a sequence $(V_l)_{l \in \N_0}$ of relatively compact open subsets of $V$ such that $\overline{V}_l \Subset V_{l+1}$ and $V = \bigcup_{l \in \N_0} V_{l}$. Set $X_l = V_l \cap X$ for $l \in \N_0$. We define
$$
C^\infty_{P,\operatorname{pol}}(V \backslash X) := \varprojlim_{l \in \N_0} \varinjlim_{N \in \N_0} \mathscr{B}_{P,N}(V_{l}\backslash X_l).
$$
We then have:
\begin{proposition} \emph{(cf.\ \cite[Satz 2.4 and Satz 3.3]{Langenbruch1978})}\label{existence-bv-distributions}
	Let $X \subseteq \R^d$ be open and let $V \subseteq \R^{d+1}$ be open such that $V \cap \R^d = X$. For all $l \in \N_0$ the mapping
	$$
	C^\infty_{P,\operatorname{pol}}(V \backslash X) \rightarrow \mathscr{D}'(X), \, f \mapsto \operatorname{bv}(D^l_t f)
	$$
	is well-defined and continuous. 
\end{proposition}
\begin{proof}
	It suffices to show that  for all $j = 0, \ldots, m-1$ the mapping
	$$
	C^\infty_{P,\operatorname{pol}}(V \backslash X) \rightarrow \mathscr{D}'(X), \, f \mapsto \operatorname{bv}(P_{(j+1)}(D) f)
	$$
	is well-defined and continuous (cf.\ the proof of Proposition \ref{existence-bv-1}). This can be shown in the same way as Proposition \ref{existence-bv} but by using Proposition \ref{almost-zero-smooth} instead of Proposition \ref{almost-zero} and Remark \ref{remark-continuity}.
\end{proof}

\section{Main results}\label{sec:Main results}
This section is devoted to the two main results of this article. Firstly, we give various characterizations  of  zero solutions of $P(D)$ that admit a boundary value in a given ultradistribution space. Secondly,  we represent the $m$-fold Cartesian product of an ultradistribution space as the quotient of certain spaces  of  zero solutions of $P(D)$.

\begin{definition}\label{def12}
	Let $M$ be a weight sequence satisfying  $(M.4)_{b_0}$ (and $p!^{1/b_0} \prec M$ in the Beurling case). Let $X \subseteq \R^d$ be open and let $V \subseteq \R^{d+1}$ be open such that $V \cap \R^d = X$. We define
	$$
	\operatorname{bv}^m_1: C^\infty_{P,[M],b_0}(V\backslash X) \rightarrow \prod_{j = 0}^{m-1}\mathscr{D}'^{[M]}(X), \, f \mapsto (\operatorname{bv}(P_{(j+1)}(D) f))_{0 \leq j \leq m-1}
	$$
	and
	$$
	\operatorname{bv}^m_2: C^\infty_{P,[M],b_0}(V\backslash X) \rightarrow \prod_{j = 0}^{m-1}\mathscr{D}'^{[M]}(X), \, f \mapsto (\operatorname{bv}(D^j_t f))_{0 \leq j \leq m-1}.
	$$
	The mappings $\operatorname{bv}^m_1$ and $\operatorname{bv}^m_2$ are well-defined and continuous by Proposition \ref{existence-bv} and Proposition \ref{existence-bv-1}, respectively. 
\end{definition}
We start by showing that $\ker \operatorname{bv}^m_n = C^\infty_P(V)$ for $n =1,2$. The proof of the next result is inspired by Komatsu's proof of the Schwarz reflection principle for ultradistributions \cite[Theorem 2.12]{Komatsu1991}.
\begin{proposition}\label{kernel-theorem}
	Let $M$ be a weight sequence satisfying  $(M.2)^*$ and $p!^{1/\gamma_0} \prec M$ ($p!^{1/\gamma_0} \subset M$). Let $X \subseteq \R^d$ be open and let $V \subseteq \R^{d+1}$ be open such that $V \cap \R^d = X$. Let $f \in C^\infty_P(V \backslash X)$ be such that $\operatorname{bv}(D^j_t f) =0$ in $\mathscr{D}'^{[M]}(X)$  for all $j = 0, \ldots, m-1$. Then, $f \in C^\infty_P(V)$. Consequently,  $\ker \operatorname{bv}^m_n = C^\infty_P(V)$ for $n =1,2$.
\end{proposition}
\begin{proof} 
	It suffices to show that for every relatively compact open subset $Y$ of $X$ there is $r>0$ such that $f \in C^\infty_P(Y \times (-r,r))$. Set $K = \overline{Y}$ and write $K_\varepsilon = K + \overline{B}(0,\varepsilon)$ for $\varepsilon > 0$. Choose $r,s  > 0$ such that $K_{2s} \times [-r,r] \Subset V$. We claim that there is $h > 0$ ($h \in \mathfrak{R}$) such that for all $\varphi \in \mathscr{D}^{M,h}_{\overline{B}(0,s)}$ 
	$$
	f \ast_x \varphi(x,t) =  \int_{\R^d} f(x-y, t) \varphi(y) dy, \qquad (x,t) \in Y \times (-r,r) \backslash Y,
	$$
	belongs to $C^\infty_P(Y \times (-r,r))$. Before we prove the claim, let us show how it entails the result. By Lemma \ref{parametrix}, there exist an ultradifferential operator $G(D_x)$ of class $[M]$, $u_1 \in  \mathscr{D}^{M,h}_{\overline{B}(0,s)}$ and $u_2 \in  \mathscr{D}^{[M]}_{\overline{B}(0,s)}$ such that $G(D_x) u_1 = \delta + u_2$ in $\mathscr{D}'^{[M]}(\R^d)$. The claim yields that $f \ast_x u_n \in C^\infty_P(Y \times (-r,r))$ for $n = 1,2$. Since $p!^{1/\gamma_0} \prec M$ ($p!^{1/\gamma_0} \subset M$), we have that  $G(D_x)(f \ast_x u_1) \in C^\infty_P(Y \times (-r,r))$ (cf.\ the proof of Lemma \ref{x-derivatives}$(ii)$). Hence,
	$$
	f = G(D_x) (f \ast_x u_1) - f \ast_x u_2 \in C^\infty_P(Y \times (-r,r)).
	$$
	We now show the claim. Let $H$ be the constant occurring in $(M.2)$. Our assumption yields that there are $T_j \in  \mathscr{D}'^{[M]}(X)$, $j = 0, \ldots, m-1$, such that
	\begin{equation}
		\lim_{t \to 0^+} D^j_t f( \, \cdot \,, t) = \lim_{t \to 0^+} D^j_t f( \, \cdot \,, -t) = T_j
		\label{limit}
	\end{equation}
	in the weak-$\ast$ topology on $(\mathscr{D}^{[M]}_{K_{2s}})'$. Since $\mathscr{D}^{[M]}_{K_{2s}}$ is an (FS)-space ((DFS)-space) and the set $\{ D^j_t f( \, \cdot \,, t) \, | \, 0 < |t| < r, j = 0, \ldots, m-1 \}$ is equicontinuous in $(\mathscr{D}^{[M]}_{K_{2s}})'$, we obtain that there is $h' > 0$ ($h' \in \mathfrak{R}$) such that  \eqref{limit} holds in the topology of uniform convergence on the unit ball of $\| \, \cdot \|_{\mathscr{E}^{M,h'}(K_{2s})}$ (in the Roumieu case we used Lemma \ref{projective-description}$(ii)$). Furthermore, by Lemma \ref{R-M2}, we may assume without loss of generality that the weight sequence $(\prod_{l=0}^p h'_l  M_p)_{p \in \N_0}$ satisfies $(M.2)$ (with  $2H$ instead of $H$). We obtain that there is a net $(c_t)_{0 < t < r}$ of positive numbers tending to zero such that for all $\psi \in \mathscr{D}^{[M]}_{K_{2s}}$ and $j = 0, \ldots, m-1$
	\begin{equation}
		\left | \int_{\R^d} D^j_t f(x,t) \psi(x) dx - \langle T_j, \psi \rangle \right | \leq c_{{|t|}} \| \psi \|_{\mathscr{E}^{M,h'}(K_{2s})}, \quad 0 < |t| < r.
		\label{convergence}
	\end{equation}
	By the Hahn-Banach theorem, we may extend  $T_j$, $j = 0, \ldots, m-1$, to a continuous linear functional on $\mathscr{D}^{M,h'}_{K_{2s}}$ (which we still denote by $T_j$). We now show that \eqref{convergence} holds for all $\psi \in \mathscr{D}^{M,h'/H}_{K_s}$ $( \psi \in \mathscr{D}^{M,h'/(2H)}_{K_s})$. Choose $\chi \in \mathscr{D}^{[M]}_{\overline{B}(0,1)}$ with $\int_{\R^d} \chi(x) dx = 1$ and set $\chi_\varepsilon(x)  = \varepsilon^{-d} \chi(x/\varepsilon)$ for $\varepsilon > 0$.  Since $\lim_{\varepsilon \to 0^+} \psi \ast \chi_\varepsilon  = \psi$ in $ \mathscr{D}^{M,h'}_{K_{2s}}$, \eqref{convergence} implies that
	\begin{align*}
		&\left | \int_{\R^d} D^j_t f(x,t) \psi(x) dx - \langle T_j, \psi \rangle \right |\\
		=& \lim_{\varepsilon \to 0^+} \left | \int_{\R^d} D^j_t f(x,t) \psi \ast \chi_\varepsilon (x) dx - \langle T_j, \psi \ast \chi_\varepsilon \rangle \right | \\
		\leq& c_{{|t|}} \lim_{\varepsilon \to 0^+} \| \psi \ast \chi_\varepsilon \|_{\mathscr{E}^{M,h'}(K_{2s})} = c_{{|t|}} \| \psi \|_{\mathscr{E}^{M,h'}(K_{2s})}.
	\end{align*}
	Set $h = h'/H^2$ ($h = h'/(2H)^2$). We are ready to prove that $f \ast_x \varphi \in C^\infty_P(Y \times (-r,r))$ for all $\varphi \in \mathscr{D}^{M,h}_{\overline{B}(0,s)}$. It is clear that $f \ast_x \varphi \in C^\infty_P(Y \times (-r,r)\backslash Y)$. Let $\alpha \in \N_0^d$ be arbitrary. Since the set $\{ D^\alpha_x\varphi(x - \, \cdot \,) \, | \, x \in Y \}$ is bounded in $\mathscr{D}^{M,h'/H}_{K_s}$ $(\mathscr{D}^{M,h'/(2H)}_{K_s})$, we infer from \eqref{convergence} that for $j = 0, \ldots, m-1$
	$$
	\lim_{t \to 0} D^\alpha_xD^j_t (f \ast_x \varphi) (x,t) = \lim_{t \to 0} \int_{\R^d} D^j_t f(y,t) D^\alpha_x\varphi(x-y) dy =  T_j \ast D^\alpha_x\varphi(x)
	$$
	uniformly for $x \in Y$. Hence, $f \ast_x \varphi$ is a function on $Y \times (-r,r)$ that is $C^\infty$ in the $x$-direction and $C^{m-1}$ in the $t$-direction. 
	Since $f \ast_x \varphi \in C^\infty_P(Y \times (-r,r)\backslash Y)$ and $Q_m = 1$, we obtain that $f \ast_x \varphi \in C^\infty_P(Y \times (-r,r))$.
\end{proof}

Proposition \ref{kernel-theorem} implies the next result; it can also be shown directly and easier by using the Schwartz parametrix method instead of Lemma \ref{parametrix}.

\begin{proposition}\label{kernel-theorem-dist} \emph{(cf.\ \cite[Satz 1.5]{Langenbruch1978})}
	Let $X \subseteq \R^d$ be open and let $V \subseteq \R^{d+1}$ be open such that $V \cap \R^d = X$. Let $f \in C^\infty_P(V \backslash X)$ be such that $\operatorname{bv}(D^j_t f) =0$ in $\mathscr{D}'(X)$  for all $j = 0, \ldots, m-1$. Then, $f \in C^\infty_P(V)$. 
\end{proposition}

Next, we show that $\prod_{j=0}^{m-1}\mathscr{E}'^{[M]}(X)$ is contained in the range of ${\bv}^m_n$ for $n= 1,2$ (provided that $a_0 = b_0$). It suffices to consider the case $X = \R^d$.

\begin{proposition}\label{local-surj} Let $M$ be a weight sequence satisfying $(M.4)_{a_0}$ and $p!^{1/\gamma_0} \prec M$ ($p!^{1/\gamma_0} \subset M$). Let $E$ be the fundamental solution of $P(D)$ constructed in Proposition \ref{theo:existence of good fundamental solutions}. For $T_0, \ldots, T_{m-1} \in  \mathscr{E}'^{[M]}(\R^d)$ we define
	\begin{align*}
		f(x,t) &= f(T_0, \ldots, T_{m-1})(x,t) \\
		&= - i \sum_{j = 0}^{m-1} \langle T_j (y), D^j_t E(x-y,t) \rangle, \qquad (x,t) \in \R^{d+1} \backslash \R^d.
	\end{align*}
	Then, $f \in  C^\infty_{P,[M],a_0}(\R^{d+1} \backslash \R^d)$ and $\operatorname{bv}^m_1(f) = (T_j)_{0 \leq j \leq m-1}$.  Consequently, for $n = 1,2$ the following statement holds: For all $T_0, \ldots, T_{m-1} \in  \mathscr{E}'^{[M]}(\R^d)$ there is $g_n \in C^\infty_{P,[M],a_0}(\R^{d+1} \backslash \R^d)$ such that  $\operatorname{bv}^m_n(g_n) = (T_j)_{0 \leq j \leq m-1}$.
\end{proposition}
We first prove the following distributional variant of Proposition \ref{local-surj}.
\begin{proposition}\label{local-surj-distributions} Let $T_0, \ldots, T_{m-1} \in \mathscr{E}'(\R^d)$. Then, $f = f(T_0, \ldots, T_{m-1}) \in  C^\infty_{P,\operatorname{pol}}(\R^{d+1} \backslash \R^d)$ and $\operatorname{bv}(P_{(j+1)}(D) f) = T_j$ in $\mathscr{D}'(\R^d)$ for $j = 0, \ldots, m-1$.
\end{proposition}
\begin{proof}
	It is clear that $f \in C^\infty_P(\R^{d+1} \backslash \R^d)$. The bound \eqref{eq:regularity} implies that there is $N \in \N_0$ such that for all $r > 0$
	$$
	\sup_{(x,t)\in\R^d \times [-r,r] \backslash \R^d} |f(x,t)| |t|^N < \infty.
	$$
	In particular, $f \in  C^\infty_{P,\operatorname{pol}}(\R^{d+1} \backslash \R^d)$. We now show that $\operatorname{bv}(P_{(j+1)}(D) f) = T_j$ in $\mathscr{D}'(\R^d)$ for $j = 0, \ldots, m-1$. We define
	$$
	T = \left ( - i \sum_{j = 0}^{m-1} T_j \otimes D^j_t \delta \right) \ast E  \in \mathscr{D}'(\R^{d+1}).
	$$
	Note that $T_{| \R^{d+1} \backslash \R^d} = f$. Let $0 \leq j \leq m-1$ and $\varphi \in \mathscr{D}(\R^d)$ be arbitrary. Set $K = \supp \varphi$. Choose $l \in \N$ and $C > 0$ such that
	$$
	|\langle T, \psi \rangle|  \leq C \| \psi \|_l, \qquad \psi \in \mathscr{D}_{K \times [-1,1]}.
	$$
	By Proposition \ref{almost-zero-smooth} (with $\check P$ instead of $P$), there exists $\Phi \in \mathscr{D}_{K\times[-1,1]}$  such that $D^l_t \Phi(\, \cdot \,, 0) = \delta_{l,j} \varphi$ for $l = 0, \ldots, m-1$ (where $\delta$ denotes the Kronecker delta) and
	$$
	\sup_{(x,t) \in \R^{d+1}  \backslash \R^d} \frac{| \check P(D) \Phi(x,t)|}{|t|^{\max\{N,l+1\}}} < \infty. 
	$$
	By employing Lemma \ref{lemma-Green} in the same way as in the proof of Proposition \ref{existence-bv}, we obtain that 
	$$
	\langle \operatorname{bv}(P_{(j+1)}(D) f), \varphi \rangle = (-1)^ji \int_\R\int_{\R^d} f(x,s) \check{P}(D)\Phi(x,s) dxds. 
	$$
	Lemma \ref{lemma:cons-taylor} implies that 
	\begin{align*}
		\int_\R\int_{\R^d} f(x,s) \check{P}(D)\Phi(x,s) dxds &= \left \langle   \left(- i \sum_{k = 0}^{m-1} {T_{k}} \otimes D^k_t \delta \right) \ast E, \check{P}(D) \Phi \right \rangle \\
		&=  - i \sum_{k = 0}^{m-1} \langle T_{k} \otimes D^k_t \delta , \Phi \rangle \\
		& = -i (-1)^j \langle T_{j},\varphi \rangle,
	\end{align*}
	which yields the result.  
\end{proof}
\begin{proof}[Proof of Proposition \ref{local-surj}]
	Since $\gamma_0 \leq a_0$, we have that $p!^{1/a_0} \prec M$ in the Beurling case. The bound  \eqref{eq:regularity} implies that for all $l \in \N_0$ it holds that $D^l_tE(\, \cdot \,, t) \in  \mathscr{E}^{\{p!^{1/a_0}\}}(\R^d) \subset \mathscr{E}^{[M]}(\R^d)$ for $t \in \R \backslash \{0\}$ fixed.  Hence, the dual pairing in the definition of $f$ makes sense. It is clear that $f \in C^\infty_P(\R^{d+1} \backslash \R^d)$.
	We now show that $f \in  C^\infty_{P,[M],a_0}(\R^{d+1} \backslash \R^d)$. By Lemma \ref{lemma-M4}$(ii)$, we have that  either $p!^{1/a_0} \prec M$ or  $p!^{1/a_0} \asymp M$.  If $p!^{1/a_0} \asymp M$, by the definition of $C^\infty_{P,[M],a_0}\left(\R^{d+1}\backslash\R^d\right)$, there is nothing left to show. Therefore, we assume that $p!^{1/a_0} \prec M$. There is $K \Subset \R^d$ such that for some $h > 0$ (for all $h > 0$) there is $C > 0$ such that for all $j = 0, \ldots, m-1$
	$$
	|\langle T_j, \varphi \rangle | \leq C \| \varphi \|_{\mathscr{E}^{M,h}(K)}, \qquad \varphi \in \mathscr{E}^{[M]}(\R^d).
	$$
	The bound  \eqref{eq:regularity} yields that there are $C_1, L_1, S > 0$ such that for all $\alpha \in \N_0^d$ and $j = 0, \ldots, m -1$
	$$
	|D^\alpha_x D^j_t E(x,t) | \leq \frac{C_1L_1^{|\alpha|} |\alpha|!^{1/a_0} }{|t|^{|\alpha|/a_0 + S}}, \qquad x \in \R^d, 0 < |t| \leq 1.
	$$
	Furthermore, $(M.2)$ implies that there are $C_2, L_2 > 0$ such that 
	$$
	\frac{2}{a_0} \omega_{M^{a_0,*}}(t) \leq \omega_{M^{a_0,*}}(L_2t) + \log C_2, \qquad t \geq 0. 
	$$
	For all $x \in \R^d$ and $t \in \R \backslash \{ 0\}$ with $|t|$ small enough  it holds that 
	\begin{align*}
		&|f(x,t)| \leq C\sum_{j = 0}^{m-1} \| D^j_tE(x- \, \cdot \, ,t) \|_{\mathscr{E}^{M,h}(K)} \leq \frac{mCC_1}{|t|^{S}} \sup_{\alpha \in \N_0^d} \frac{L_1^{|\alpha|} |\alpha|!^{1/a_0}}{h^{|\alpha|}M_\alpha|t|^{|\alpha|/a_0}} \\
		& \leq \frac{mCC_1}{|t|^{S}} \sup_{\alpha \in \N_0^d} \left (\frac{(L_1/h)^{a_0|\alpha|}}{M^{a_0,*}_\alpha |t|^{|\alpha|}} \right)^{1/a_0} = mCC_1e^{\frac{1}{a_0}\omega_{M^{a_0,\ast}}\left( \frac{1}{(h/L_1)^{a_0} |t| }\right) + S \log \left( \frac{1}{|t|} \right) } \\
		&\leq  mCC_1e^{\frac{2}{a_0}\omega_{M^{a_0,\ast}}\left( \frac{1}{(h/L_1)^{a_0} |t| }\right)} \leq mCC_1C_2e^{\omega_{M^{a_0,\ast}}\left( \frac{L_2}{(h/L_1)^{a_0} |t| }\right)}
	\end{align*}
	for some $h > 0$ (for all $h > 0$). This shows that $f \in  C^\infty_{P,[M],a_0}(\R^{d+1} \backslash \R^d)$. Finally, we prove that $\operatorname{bv}^m_1(f) = (T_j)_{0 \leq j \leq m-1}$. By Lemma \ref{parametrix} (and Lemma \ref{projective-description}$(i)$ in the Roumieu case), there are an ultradifferential operator $G(D_x)$ of class $[M]$ and $T_{q,j} \in C_c(\R^d) \subset \mathscr{E}'(\R^d)$ for $q=1,2$ and $j = 0, \ldots, m-1$ such that $T_j = G(D_x) T_{1,j} + T_{2,j}$ in $\mathscr{D}'^{[M]}(\R^d)$. Since
	$$
	f(T_0, \ldots, T_{{m-1}}) =  G(D_x) f(T_{1,0}, \ldots, T_{1,{m-1}}) + f(T_{2,0}, \ldots, T_{2,{m-1}}),
	$$
	the result follows from Lemma \ref{x-derivatives} and Proposition \ref{local-surj-distributions}. The result obviously shows that for all $T_0, \ldots, T_{m-1} \in  \mathscr{E}'^{[M]}(\R^d)$ there is $g \in C^\infty_{P,[M],a_0}(\R^{d+1} \backslash \R^d)$ such that  $\operatorname{bv}^m_1(g) = (T_j)_{0 \leq j \leq m-1}$. The corresponding statement for $\operatorname{bv}^m_2$ follows from it by \eqref{base-change}.
\end{proof}
We are ready to prove the two main results of this article. 
\begin{theorem}\label{main-1}
	Suppose that $a_0 = b_0$. Let $M$ be a weight sequence satisfying $(M.4)_{a_0}$  and $p!^{1/\gamma_0} \prec M$ ($p!^{1/\gamma_0} \subset M$). Let $X \subseteq \R^d$ be open and let $V \subseteq \R^{d+1}$ be open  such that $V \cap \R^d = X$. Let $f \in C^\infty_P(V\backslash X)$. The following statements are equivalent:
	\begin{itemize}
		\item[$(i)$]  $f \in C^\infty_{P,[M],a_0}(V\backslash X)$.
		\item[$(ii)$] $\operatorname{bv}(D^l_tf) \in \mathscr{D}'^{[M]}(X)$ exists for all $l \in \N_0$.
		\item[$(iii)$] $\operatorname{bv}(f) \in \mathscr{D}'^{[M]}(X)$ exists.
		\item[$(iv)$] For every relatively compact  open subset $Y$ of $X$ there is $r>0$ such that $\{ f(\, \cdot \,, t) \, | \, 0 < |t| < r \}$  is bounded in  $\mathscr{D}'^{[M]}(Y)$.
		\item[$(v)$] For every relatively compact  open subset $Y$ of $X$ there is $r > 0$ such that
		$f = G(D_x) f_1 + f_2$ on $Y \times (-r,r) \backslash Y $, where $G(D_x)$ is an ultradifferential operator of class $[M]$ and $f_q \in  C^\infty_P(Y \times (-r,r) \backslash Y) \cap L^\infty(Y \times (-r,r))$ for $q = 1,2$.
	\end{itemize}
\end{theorem}
\begin{proof}
	$(i) \Rightarrow (ii)$ This has been shown in Proposition \ref{existence-bv-1}.
	
	$(ii) \Rightarrow (i)$ It suffices to show that for  every relatively compact  open subset $Y$ of $X$ and $r>0$ with $\overline{Y} \times [-r,r] \Subset V$ it holds that $f \in C^\infty_{P,[M],a_0}(Y \times (-r,r) \backslash Y)$.  Choose $\psi \in \mathscr{D}^{[M]}(X)$ such that $\psi = 1$ on $Y$. Set $T_j = \psi \bv(D^j_t f) \in \mathscr{E}'^{[M]}(\R^d)$ for $j = 0, \ldots, m-1$. By Proposition \ref{local-surj}, there is $g \in C^\infty_{P,[M],a_0}(\R^{d+1} \backslash \R^d)$ such that $\bv(D^j_t g) = T_j$ in  $\mathscr{D}'^{[M]}(\R^d)$ and thus  $\bv(D^j_t g) = T_j =  \bv(D^j_t f)$ in  $\mathscr{D}'^{[M]}(Y)$ for $j = 0, \ldots, m-1$. Proposition \ref{kernel-theorem} implies that  $f-g \in C^\infty_P(Y \times (-r,r))$, whence $f \in C^\infty_{P,[M],a_0}(Y \times (-r,r) \backslash Y)$.
	
	$(ii) \Rightarrow (iii)$ \emph{and} $(iii) \Rightarrow (iv)$  Obvious. 
	
	$(iv) \Rightarrow (v)$  This can be shown by using the parametrix method (Proposition \ref{parametrix}) in a similar way as in the proof of Proposition \ref{kernel-theorem}.
	
	$(v) \Rightarrow (ii)$ This follows from Lemma \ref{x-derivatives}$(ii)$ and Proposition \ref{existence-bv-1} as obviously $f_q\in C_{P,[M],a_0}^\infty(Y\times(-r,r)\backslash Y)$ for $q = 1,2$.
\end{proof}
\begin{theorem}\label{main-2}
	Suppose that $a_0 = b_0$. Let $M$ be a weight sequence satisfying $(M.4)_{a_0}$ and $p!^{1/\gamma_0} \prec M$ ($p!^{1/\gamma_0} \subset M$).  Let $X \subseteq \R^d$ be open and let $V \subseteq \R^{d+1}$ be open such that $V \cap \R^d = X$. Suppose that $V$ is $P$-convex for supports. For $n = 1,2$ the sequence
	$$
	0  \xrightarrow{\phantom{\phantom,\operatorname{bv}^m_n\phantom,}} C^\infty_P(V)  \xrightarrow{\phantom{\phantom,\operatorname{bv}^m_i\phantom,}} C^\infty_{P,[M],a_0}(V\backslash X) \xrightarrow{\phantom,\operatorname{bv}^m_n\phantom,} \prod_{j = 0}^{m-1}\mathscr{D}'^{[M]}(X)  \xrightarrow{\phantom{\phantom,\operatorname{bv}^m_n\phantom,}} 0
	$$
	is exact and $\operatorname{bv}^m_n$ is a topological homomorphism.
\end{theorem}
\begin{proof}
	The mappings $\operatorname{bv}^m_1$ and $\operatorname{bv}^m_2$ are continuous by Proposition \ref{existence-bv} and Proposition \ref{existence-bv-1}, respectively. Proposition \ref{kernel-theorem} yields that  $\ker \operatorname{bv}^m_n = C^\infty_P(V)$. Next, we show that  $\operatorname{bv}^m_n$ is surjective. To this end, we shall use some basic facts about the derived projective limit functor; we refer to the book \cite{Wengenroth} for more information. Choose a sequence $(V_l)_{l \in \N_0}$ of relatively compact  open subsets of $V$ such that $\overline{V}_l \Subset V_{l+1}$ and $V = \bigcup_{l \in \N_0} V_l$. Set $X_l = V_l \cap X$ and $\operatorname{bv}^m_{n,l} = \operatorname{bv}^m_{n} :  C^\infty_{P,[M],a_0}(V_l\backslash X_l) \rightarrow \prod_{j = 0}^{m-1}\mathscr{D}'^{[M]}(X_l)$ for $l \in \N_0$. We define the projective spectra
	$$
	\mathscr{X} = (C^\infty_P(V_l))_{l \in \N_0}, \; \mathscr{Y} = (C^\infty_{P,[M],a_0}(V_l\backslash X_l))_{l \in \N_0}, \; \mathscr{Z} = \left( \prod_{j = 0}^{m-1}\mathscr{D}'^{[M]}(X_l) \right)_{l \in \N_0},
	$$
	(with restriction as spectral mappings) and the morphism $(\operatorname{bv}^m_{n,l})_{l \in \N_0}:  \mathscr{Y} \rightarrow  \mathscr{Z}$. 
	We need to show that the mapping
	$$
	\operatorname{bv}^m_n = \operatorname{Proj} \, (\operatorname{bv}^m_{n,l})_{l \in \N_0} : C^\infty_{P,[M],a_0}(V\backslash X) = \operatorname{Proj} \, \mathscr{Y} \rightarrow \prod_{j = 0}^{m-1}\mathscr{D}'^{[M]}(X) = \operatorname{Proj} \, \mathscr{Z} 
	$$
	is surjective. Proposition \ref{kernel-theorem} yields  that $\mathscr{X} = \ker (\operatorname{bv}^m_{n,l})_{l \in \N_0}$. By \cite[Proposition 3.1.8]{Wengenroth}, it suffices to show:
	\begin{itemize}
		\item[$(i)$] For all $l \in \N_0$ and $T_0, \ldots, T_{m-1} \in \mathscr{D}'^{[M]}(X_{l+1})$ there is $f \in C^\infty_{P,[M],a_0}(V_l\backslash X_l)$ such that $\operatorname{bv}^m_{n,l}(f) = (T_{0| X_l}, \ldots, T_{m-1| X_{l}})$.
		\item[$(ii)$] $\operatorname{Proj}^1 \, \mathscr{X} = 0$. 
	\end{itemize}
	Property $(i)$ is a consequence of Proposition \ref{local-surj} (cf.\ the proof of  $(ii) \Rightarrow (i)$  in Theorem \ref{main-1}), whereas $(ii)$ follows from the fact that $V$ is $P$-convex for supports \cite[Section 3.4.4]{Wengenroth}. 
	Finally, since $C^\infty_{P,[M],a_0}(V\backslash X)$ is webbed and  $\prod_{j = 0}^{m-1}\mathscr{D}'^{[M]}(X)$ is ultrabornological,  $\operatorname{bv}^m_n$ is a topological homomorphism by the open mapping theorem of De Wilde.
\end{proof}
We end this section by giving the analogues of Theorem \ref{main-1} and Theorem \ref{main-2} for distributions.

\begin{theorem} \emph{(cf.\ \cite[Satz 2.4]{Langenbruch1978})} 
	Let $X \subseteq \R^d$ be open and let $V \subseteq \R^{d+1}$ be open such that $V \cap \R^d = X$. Let $f \in C^\infty_P(V\backslash X)$. The following statements are equivalent:
	\begin{itemize}
		\item[$(i)$]  $f \in C^\infty_{P, \operatorname{pol}}(V\backslash X)$.
		\item[$(ii)$] $\operatorname{bv}(D^l_tf) \in \mathscr{D}'(X)$ exists for all $l \in \N_0$.
		\item[$(iii)$] $\operatorname{bv}(f) \in \mathscr{D}'(X)$ exists.
		\item[$(iv)$] For every relatively compact  open subset $Y$ of $X$ there is $r>0$ such that $\{ f(\, \cdot \,, t) \, | \, 0 < |t| < r \}$  is bounded in  $\mathscr{D}'(Y)$.
		\item[$(v)$] For every relatively compact  open subset $Y$ of $X$ there is $l \in \N_0$ such that
		$f = \Delta^l_x f_1 + f_2$ on $Y \times (-r,r) \backslash Y $, where $f_q \in  C^\infty_P(Y \times (-r,r) \backslash Y) \cap L^\infty(Y \times (-r,r))$ for $q = 1,2$.
	\end{itemize}
\end{theorem}
\begin{proof}
	This can be shown in a similar way as Theorem \ref{main-1}.
\end{proof}
The mappings $\operatorname{bv}^m_1$ and $\operatorname{bv}^m_2$ for distributions are defined similarly as in Definition \ref{def12}. We then have:
\begin{theorem}  \emph{(cf.\ \cite[Satz 3.3 and the subsequent remark]{Langenbruch1978})}
	Let $X \subseteq \R^d$ be open and let $V \subseteq \R^{d+1}$ be open such that $V \cap \R^d = X$. Suppose that $V$ is $P$-convex for supports. For $n = 1,2$ the sequence
	$$
	0  \xrightarrow{\phantom{\phantom,\operatorname{bv}^m_n\phantom,}} C^\infty_P(V)  \xrightarrow{\phantom{\phantom,\operatorname{bv}^m_i\phantom,}} C^\infty_{P,\operatorname{pol}}(V\backslash X) \xrightarrow{\phantom,\operatorname{bv}^m_n\phantom,} \prod_{j = 0}^{m-1}\mathscr{D}'(X)  \xrightarrow{\phantom{\phantom,\operatorname{bv}^m_n\phantom,}} 0
	$$
	is exact and $\operatorname{bv}^m_n$ is a topological homomorphism.
\end{theorem}
\begin{proof}
	This can be shown in a similar way as Theorem \ref{main-2}.
\end{proof}

\section{Semi-elliptic polynomials}\label{sect-examples} 
In this section we discuss the assumptions of our main results (Theorem \ref{main-1} and Theorem \ref{main-2}) for semi-elliptic polynomials. For $\beta\in\N_0^{d+1}$ and ${\bf{n}}  \in\N^{d+1}$ we write $|\beta:{\bf{n}}|:=\sum_{j=1}^{d+1}\frac{\beta_j}{n_j}$. 
\begin{definition}
	A polynomial $P\in\C[\xi_1,\ldots,\xi_{d+1}]$ is called \textit{semi-elliptic} if there is  ${\bf{n}} \in\N^{d+1}$ such that $P$ can be written as 
	$$
	P(\xi)=\sum_{|\beta: {\bf{n}}|\leq 1}c_\beta\xi^\beta
	$$ 
	and 
	$$
	P^0(\xi):=\sum_{|\beta: {\bf{n}}|=1}c_\beta\xi^\beta\neq 0
	$$
	for all $\xi\in\R^{d+1}\backslash\{0\}$. 
\end{definition}
If $P$ is semi-elliptic, the multi-index ${\bf{n}}$ in the above definition is unique. More precisely, $n_j=\operatorname{deg}_{\xi_j}P$, $j= 1, \ldots, d+1$ \cite[Lemma 7.14]{Treves}. Furthermore,  $\deg P = \max\{\operatorname{deg}_{x_j}P\,|\, 1\leq j\leq d+1\}$ \cite[Proposition 2]{FrKa}. It is well-known that semi-elliptic polynomials are hypoelliptic; see \cite[Theorem 11.1.11]{HoermanderPDO2} and \cite[Theorem 7.7]{Treves}.  

We now determine the values $a_0(P)$, $b_0(P)$ and $\gamma_0(P)$ for semi-elliptic polynomials $P$.
\begin{proposition}\label{prop:semi-elliptic a_0 and b_0}
	Let $P(x,t)=\sum_{k=0}^m Q_k(x)t^k$ be semi-elliptic. Then, 
	\begin{align*}
		&a_0(P)=\frac{\min\{\operatorname{deg}_{x_j}P\,|\, 1\leq j\leq d\}}{m}, \\
		&b_0(P) = \frac{\max\{\operatorname{deg}_{x_j}P\,|\, 1\leq j\leq d\}}{m} = \frac{\deg Q_0 }{m}, \\
		&\gamma_0(P)=\frac{\min\{\operatorname{deg}_{x_j}P\,|\, 1\leq j\leq d\}}{\operatorname{deg} P}.
	\end{align*}
\end{proposition}
\begin{proof}
	The result for $\gamma_0$ follows from \cite[Theorem 11.4.15]{HoermanderPDO2} (see also \cite[Theorem 7.7]{Treves}). Set $n_j = \deg_{x_j} P$, $j=1, \ldots ,d$, $n_{d+1} = m$ and ${\bf{n}} = (n_1, \ldots, n_{d+1})$. Then, $P(\xi)=\sum_{|\beta: {\bf{n}}|\leq 1}c_\beta\xi^\beta$, $\xi \in \R^{d+1}$, for suitable $c_\beta \in \C$ and 
	$P^0(\xi)=\sum_{|\beta: {\bf{n}}|=1}c_\beta \xi^\beta\neq 0$ for all $\xi\in\R^{d+1}\backslash\{0\}$. We define $|\xi|_{{\bf{n}}} =\sum_{j=1}^{d+1}|\xi_j|^{n_j}, \xi\in\R^{d+1}$.  Note that for all $\beta \in \N_0^{d+1}$
	\begin{equation}\label{eq:semi-elliptic 1}
		\,|\xi^\beta|\leq |\xi|_{{\bf{n}}}^{|\beta: {\bf{n}}|}, \qquad \xi\in\R^{d+1}.
	\end{equation}
	Furthermore, there are $C,R \geq 1$ such that for all  $\beta\in\N_0^{d+1}$  
	\begin{equation}\label{eq:semi-elliptic 5}
		|\xi|_{{\bf{n}}}^{|\beta: {\bf{n}}|}|D^\beta P(\xi)|\leq C|P(\xi)|, \qquad |\xi| \geq R,
	\end{equation}
	as essentially follows from \eqref{eq:semi-elliptic 1} and the fact that $P$ is semi-elliptic; see the proof of \cite[Theorem 7.7]{Treves} for details. We now determine  $a_0$. For all $x\in\R^d$ with $|x| \geq R$, $t \in \R$ and $l \in \N_0$ it holds that
	\begin{align*}
		|(x,t)|_{{\bf{n}}}^{|(0, \ldots, 0,l): {\bf{n}}|} &\geq\left(\sum_{j=1}^d |x_j|^{n_j}\right)^\frac{l}{m}\geq \left(\sum_{j=1}^d\left(\frac{|x_j|}{|x|}\right)^{n_j} |x|^{\min\{n_j \,| \, 1 \leq j\leq d\}}\right)^\frac{l}{m}\\
		&\geq \inf_{y\in\R^d, |y|=1}\left(\sum_{j=1}^d |y_j|^{n_j}\right)^\frac{l}{m}|x|^{l \frac{\min\{n_j \, | \,1\leq j\leq d\}}{m}},
	\end{align*}
	which, by \eqref{eq:semi-elliptic 5}, shows that $a_0\geq \min\{n_j\, | \,1\leq j\leq d\}/m$. For the converse inequality, note that
	\begin{equation}
		Q_0(x) = \sum_{|(\alpha,0) : {\bf{n}}| \leq 1} c_{(\alpha,0)} x^{\alpha}.
		\label{Q0-repr}
	\end{equation}
	Hence, \eqref{eq:semi-elliptic 1} implies that 
	$$
	|Q_0(x)|\leq C_0\sum_{j=1}^d |x_j|^{n_j}, \qquad x \in \R^d, |x| \geq 1,
	$$
	for some $C_0 > 0$.  Inequality \eqref{eq:semi-elliptic 5} particularly  shows that 
	$$
	|\xi|_{{\bf{n}}}\leq C_1|P(\xi)|, \qquad |\xi| \geq R,
	$$
	for some $C_1 > 0$. The above two inequalities together with \eqref{eq:Q_0 dominates} yield that  for $x\in\R^d$ with $|x|$ sufficiently large
	$$\sum_{j=1}^d |x_j|^{n_j}+|x|^{a_0 m}=|(x,|x|^{a_0})|_{{\bf{n}}}\leq C_1 |P(x,|x|^{a_0})|\leq C\sum_{j=1}^d |x_j|^{n_j}$$
	for some $C > 0$, which implies that $a_0\leq \min\{n_j\, | \,1\leq j\leq d\}/m$. This proves the result for $a_0$. Next, we determine $b_0$.  Let $0 \leq k \leq m-1$ with $Q_k \neq 0$ be arbitrary. Then, there is $\alpha \in \N^d$ with $|\alpha| = \deg Q_k$ such that $c_{(\alpha,k)} \neq 0$ and, thus, $|(\alpha,k) : {\bf{n}}| \leq 1$. Hence,
	$$
	1\geq  \sum_{j=1}^d\frac{\alpha_j}{n_j}+\frac{k}{m} \geq \frac{\deg Q_k}{\max \{n_j\, | \,1\leq j\leq d\}}+\frac{k}{m}.
	$$
	This implies that  $\operatorname{deg}Q_k/(m-k)\leq \max \{n_j\, | \,1\leq j\leq d\}/m$ and, thus, $b_0\leq  \max \{n_j\, | \,1\leq j\leq d\}/m$. Let $1 \leq j \leq d$ be arbitrary. Consider $\alpha^j  =(0, \ldots, \alpha^j_j = n_j, \ldots, 0) \ \in \N^d$. Then,  $c_{(\alpha^j,0)} \neq 0$ for otherwise $(\delta_{l,j})_{1 \leq l \leq d+1} \neq 0$ would be a zero of $P^0$. By \eqref{Q0-repr}, we obtain that $n_j = \deg_{x_j} Q_0$. Since $Q_0$ is semi-elliptic (on $\R^d$)  as $P$ is so, we have that $\operatorname{deg}Q_0=\max\{\deg_{x_j} Q_0 \,| \, 1\leq j\leq d\}$.  Hence,  $\max \{n_j\, | \,1\leq j\leq d\}/m = \deg Q_0/m$. Clearly,  $\deg Q_0/m \leq b_0$. This proves the result for $b_0$.
\end{proof}
\begin{remark}
	Let $P(x,t)=\sum_{k=0}^m Q_k(x)t^k$ be semi-elliptic. By \cite[Theorem 11.4.15]{HoermanderPDO2}, we have that $\mu_0(P)=m/\operatorname{deg}P$. Since the exact value of $\mu_0(P)$ is not important to us, we do not consider it in what follows.
\end{remark}
\begin{corollary}\label{cor:semi-elliptic a_0 and b_0} 
	Let $P(x,t)=\sum_{k=0}^m Q_k(x)t^k$ be a semi-elliptic polynomial with $\operatorname{deg}_{x_1}P=\cdots=\operatorname{deg}_{x_d}P=n$. Then, 
	$$
	a_0(P)= b_0(P) = \frac{n}{m} =\frac{\deg Q_0 }{m} \qquad \mbox{and} \qquad \gamma_0(P)= \min \left\{ 1, \frac{n}{m} \right\}.
	$$
\end{corollary}

Next,  given a semi-elliptic polynomial $P$, we present a sufficient geometric condition on an open set $V \subseteq \R^{d+1}$ to be $P$-convex for supports \cite{FrKa,Kalmes19-1}.
A function $f : V \rightarrow \R$ is said to satisfy the minimum principle in a closed subset $F$ of $\R^{d+1}$ if for all $K \Subset F \cap V$ it holds that
$$
\min_{x \in K} f(x) = \min_{x \in \partial_F K} f(x),
$$
where $\partial_F K$ denotes the boundary of $K$ in $F$ {(in case of $F\cap V=\emptyset$ the condition is vacuous).}
\begin{proposition}\label{P-convex-suff}  (\cite[Theorem 1]{Kalmes19-1} \emph{and} \cite[Proposition 2$(iii)$]{FrKa}) Let $P$ be a semi-elliptic polynomial on $\R^{d+1}$. Then, an open set $V\subseteq \R^{d+1}$ is $P$-convex for supports if its boundary distance function
	$$
	d_V \, : \, V\rightarrow \R, \, \xi \mapsto d(\xi, \R^{d+1} \backslash V),
	$$
	where $d(\xi,\R^{d+1}\backslash V)=\inf\{|\xi-\eta| \, | \,\eta\in\R^{d+1}\backslash V\}$, satisfies the minimum principle in every affine subspace parallel to the subspace
	$$
	\{\xi\in\R^{d+1} | \xi_j=0\mbox{ for all } 1\leq j \leq d+1 \mbox{ with }\operatorname{deg}_{\xi_j}P<\operatorname{deg}P\}.
	$$
\end{proposition}
Corollary \ref{cor:semi-elliptic a_0 and b_0} implies that Theorem \ref{main-1} and Theorem \ref{main-2} are applicable to every semi-elliptic polynomial with $\operatorname{deg}_{x_1}P=\cdots=\operatorname{deg}_{x_d}P= n$. Moreover, Proposition \ref{P-convex-suff} gives a sufficient  condition on the open sets $V \subset \R^{d+1}$ for which Theorem \ref{main-2} is valid. We now discuss this class of polynomials in some more detail. We distinguish three cases:

$(i)$ $m = n = \deg P$:  Then, $P$ is elliptic and, conversely, every elliptic polynomial is of this form. Corollary \ref{cor:semi-elliptic a_0 and b_0} implies that 
$$
a_0(P)= b_0(P) = \gamma_0(P) = 1,
$$
while Proposition \ref{P-convex-suff} yields the well-known fact that every open set $V \subseteq \R^{d+1}$ is $P$-convex for supports.

$(ii)$ $m < n = \deg P$:  The heat operator and, more generally, the $k$-parabolic operators in the sense of Petrowsky  \cite[Definition 7.11]{Treves} are of this form.
Corollary \ref{cor:semi-elliptic a_0 and b_0} implies that 
$$
a_0(P)= b_0(P) = n/m  \qquad \mbox{and} \qquad \gamma_0(P) = 1.
$$
By Proposition \ref{P-convex-suff}, we have that an open set $V \subseteq \R^{d+1}$ is $P$-convex for supports if $d_V$ satisfies the minimum principle in every characteristic hyperplane, i.e., in every hyperplane of the form $\{(x,\tau)|\,x\in\R^d\}, \tau\in\R$. We mention that, by \cite[Corollary 5]{Kalmes19-1} (see also \cite[Theorem 10.8.1]{HoermanderPDO2}), the converse holds true as well in this case. In particular,  all open sets of the form $V= X\times I$, $X\subseteq\R^d$ and $I\subseteq\R$ open, are $P$-convex for supports.

$(iii)$ $m = \deg P > n$: Corollary \ref{cor:semi-elliptic a_0 and b_0} implies that 
$$
a_0(P)= b_0(P) = \gamma_0(P) =  n/m.
$$ 
By Proposition \ref{P-convex-suff}, we have that an open set $V \subseteq \R^{d+1}$ is $P$-convex for supports if $d_V$ satisfies the minimum principle in every line $\{(x,\tau)|\,\tau \in \R\}, x\in\R^d$. Again, all open sets of the form $V= X\times I$, $X\subseteq\R^d$ and $I\subseteq\R$ open, are $P$-convex for supports.

Finally, note that Theorem \ref{cor:semi-elliptic operators} stated in the introduction follows from Theorem \ref{main-1}, Theorem \ref{main-2} and the above remarks.

\bibliographystyle{plain}

\end{document}